\documentclass[11pt]{article}
\usepackage[utf8]{inputenc}
\usepackage[T1]{fontenc}
\usepackage{graphicx,amsmath,amsfonts,amssymb,amsthm,psfrag, color,framed}
\usepackage{comment}
\usepackage{calc}
\usepackage{wrapfig,float}

\graphicspath{{file/}}

\includecomment{comment} 

\newtheorem{thm}{Theorem}[section]

\newtheorem{lemma}[thm]{Lemma}
\newtheorem{proposition}[thm]{Proposition}
\newtheorem{assumption}{Assumption}[section]
\newtheorem{definition}[thm]{Definition}

\newtheorem{rem}[thm]{Remark}
\numberwithin{equation}{section}

\title{Approximation of exit times for one-dimensional linear and growth diffusion processes.}

\begin{document}

\author{S. Herrmann and N. Massin\\[5pt]
\small{Institut de Math{\'e}matiques de Bourgogne (IMB) - UMR 5584, CNRS,}\\
\small{Universit{\'e} de Bourgogne Franche-Comt\'e, F-21000 Dijon, France} \\
\small{Samuel.Herrmann@u-bourgogne.fr}\\
\small{Nicolas.Massin@u-bourgogne.fr}
}
\maketitle

\begin{abstract}
In order to approximate the exit time of a one-dimensional diffusion process, we propose an algorithm based on a random walk. Such an algorithm was already introduced in both the Brownian context and the Ornstein-Uhlenbeck context, that is for particular time-homogeneous diffusion processes. Here the aim is therefore to generalize this efficient numerical approach in order to obtain an approximation of both the exit time and position for either a general linear diffusion or a growth diffusion. The main challenge of such a generalization is to handle with time-inhomogeneous diffusions. The efficiency of the method is described with particular care through theoretical results and numerical examples.
\end{abstract}
\textbf{Key words and phrases:} exit time, linear diffusion, growth diffusion, random walk, generalized spheroids, stochastic algorithm\\
\textbf{2010 AMS subject classifications:} primary: 65C05; secondary: 60J60, 60G40, 60G46.
\section{Introduction}
In many domains, the simulation of the first exit time for a diffusion plays a crucial role. In reliability analysis, for instance, first passage times and exit times are directly related to lifetimes of engineering systems. In order to emphasize explicit expressions of the lifetime distribution, it is quite usual to deal with simplified models like Ornstein-Uhlenbeck processes. Indeed they satisfy the mean reverting property which is essential for modeling degradation processes. In mathematical finance studying barrier options also requires to describe exit times since it is of prime interest to estimate if the underlying stock price stays in a given interval. 
In the simple Black-Scholes model, the distribution of the first exit time is well-known. In more complex models corresponding to general diffusion processes, such an explicit expression is not available and requires the use of numerical approximations.

Several methods have been introduced in order to approximate first exit times. The classical and most common approximation method is the Euler–Maruyama scheme based on a time-discretization procedure. The exit time of the diffusion process is in that case replaced by the exit time of the scheme. The approximation is quite precise but requires to restrict the study on a given fixed time interval on one hand and to describe precisely the probability for the diffusion to exit inbetween two consecutive nodes of the time grid on the other hand.
In this study, we aim to introduce a random walk in order to approximate the diffusion exit time from a given interval. Let us introduce $(X_t,\, t\ge 0)$ the unique solution of the stochastic differential equation:
\[
dX_t=b(t,X_t)\,dt+\sigma(t,X_t)\,dW_t,\quad t\ge 0,
\]
where $(W_t,\, t\ge 0)$ stands for a one-dimensional Brownian motion. Let us also fix some interval $I=[a,b]$ which strictly contains the starting position $X_0=x$. We denote by $\mathcal{T}$ the diffusion first exit time:
\[
\mathcal{T}=\inf\{t\ge 0:\ X_t\notin [a,b]\}.
\]
Our approach consists in constructing a random walk $(T_n,X_n)_{n\ge 0}$ on $\mathbb{R}_+\times \mathbb{R}$ which corresponds to a skeleton of the Brownian paths. In other words, the sequence $(T_n,X_n)$ belongs to the graph of the trajectory. Moreover we construct the walk in such a way that $(T_n,X_n)$ converges as time elapses towards the exit time and location $(\mathcal{T},X_{\mathcal{T}})$. It suffices therefore to introduce a stopping procedure in the algorithm to achieve the approximation scheme. Of course, such an approach is interesting provided that $(T_n,X_n)$ is easy to simulate numerically. For the particular Brownian case, the distribution of the exit time from an interval has a quite complicated expression which is difficult to use for simulation purposes (see, for instance \cite{sacerdote-telve-zucca}) whereas the exit distribution from particular time-dependent domains, for instance
the spheroids also called \emph{heat balls}, can be precisely determined. These time-dependent domains are characterized by their boundaries: 
 \begin{equation}
 \psi_\pm(t) = \pm\sqrt{t\log\left(\frac{d^2}{t}\right)}, \quad \text{for } t \in [0,d^2],
 \label{sphbm}
 \end{equation}
where the parameter $d>0$ corresponds to the size of the spheroid. The first time the Brownian motion paths $(t, W_t)$ exits from the domain\\ 
$\{(t,x):\ |x|\le \psi_+(t)\}$, denoted by $\tau$, is well-known. Its probability density function \cite{lerche} is given by
\begin{equation}
p(t)= \frac{1}{d \sqrt{2 \pi}} \sqrt{\frac{1}{t}\log\left(\frac{d^2}{t}\right)},\quad t\ge 0.
\end{equation}
It is therefore easy to generate such an exit time since $\tau$ and $d^2Ue^{-N^2}$are identically distributed. Here $U$ and $N$ are independent random variables,  $U$ being uniformly distributed on $[0,1]$ and $N$ being a standard gaussian random variable.
Let us notice that the boundaries of the spheroids satisfy the following bound:
\begin{equation}
\vert \psi_\pm(t)\vert \leqslant \frac{d}{\sqrt{e}}, \quad \forall t\in [0,d^2].
\label{boundmb}
\end{equation}
This remark permits to explain the general idea of the algorithm. First we consider $(T_0,X_0)$ the starting time and position of the Brownian paths, that is $(0,x)$. Then we choose the largest parameter $d$ possible such that the spheroid starting in $(T_0,X_0)$ is included in the domain $\mathbb{R}_+\times [a,b]$. We observe the first exit time of this spheroid and its corresponding exit location, this couple is denoted by $(T_1,X_1)$. Due to the translation invariance of the Brownian motion, we can construct an iterative procedure, just considering $(T_1,X_1)$ like a starting time and position for the Brownian motion. So we consider a new spheroid included in the interval and $(T_2,X_2)$ shall correspond to the exit of this second spheroid and so on. Step by step we construct a random walk on spheroids also called WOMS algorithm (Walk On Moving Spheres) which converges towards the exit time and position $(\mathcal{T},W_{\mathcal{T}})$. This sequence is stopped as soon as the position $X_n$ is
close enough to the boundary of the considered interval. The idea of this algorithm lies in the definition of spherical processes and the walk on spheres introduced by Müller \cite{mul} and used in the sequel by Motoo \cite{mot} and Sabelfeld \cite{sab1} \cite{sab2}. It permits also in some more technical advanced way to simulate the first passage time for Bessel processes \cite{Deaconu-Herrmann}.

In this study, we focus our attention on diffusions which are strongly related to the Brownian motion: they can be expressed as functionals of the Brownian motion that is $X_t=f(t,W_t)$. The
idea is to use this link to adapt the Brownian algorithm in an appropriate way. This link implies changes on the time-dependent domains for which the exit problem can be expressed in a simpler way.
For these diffusion families, we present the random walk algorithm (WOMS), describe the approximation error depending on the stopping procedure and emphasize the efficiency of the method. We describe the mean number of generalized spheroids necessary to obtain the approximated exit time.

\section{WOMS algorithm for L-class diffusions}
The Walk on Spheroids already introduced for the Ornstein-Uhlenbeck process in \cite{Herrmann-Massin-1} permits to approximate the exit time in an efficient way. We aim to extend such numerical procedure to a wider  class of stochastic processes. We focus our attention to the family of L-class diffusions (linear-type diffusions) which generalizes the Ornstein-Uhlenbeck processes. For such diffusions, all the coefficients are time-dependent. Moreover they are based on a strong relation with a one-dimensional Brownian motion.

\subsection{L-class diffusions}
This particular family of diffusions was already introduced in \cite{potzelberger-wang}.
\begin{definition}[L-class diffusions]
We call L-class diffusion any solution of 
\begin{equation}
dX_t = (\alpha(t)X_t + \beta(t))dt + \tilde{\sigma}(t)dW_t\quad t\ge 0,
\label{edsl}
\end{equation}
 where $\alpha$ and $\beta$ are real continuous functions, $\tilde{\sigma}$ is a continuous non-negative function and $(W_t)_{ t\ge 0}$ is a one-dimensional Brownian motion.
\end{definition}

We solve equation \eqref{edsl} in a classical way. Let us introduce 
\begin{equation}
\theta(t):=-\displaystyle{\int^t_0 \alpha(s)ds}.
\label{theta}
\end{equation}

\begin{lemma}
The unique solution of \eqref{edsl} is given by
\begin{equation*}
X_t = X_0\, e^{-\theta(t)} + e^{-\theta (t) }\displaystyle{\int^t_0 e^{\theta(s)} \beta(s) ds}+ e^{-\theta(t)} \displaystyle{\int^t_0 e^{\theta(s)} \tilde{\sigma}(s)dW_s} ,\quad t\ge 0.
\end{equation*}
\end{lemma}

\begin{proof}
Let us consider $g(t,x) = xe^{\theta(t)}$. The statement is therefore an easy consequence of Itô's formula:
\begin{align*}
d(g(t,X_t)) &=-\alpha(t) X_t e^{\theta(t)} dt + e^{\theta(t)}dX_t \\
&= -\alpha(t) X_t e^{\theta(t)} dt+ e^{\theta(t)}(\alpha(t) X_t + \beta(t))dt + e^{\theta(t)} \tilde{\sigma}(t)dW_t\\
&= e^{\theta(t)}\beta(t)dt + e^{\theta(t)}\tilde{\sigma}(t)dW_t, \quad t\ge 0.
\end{align*} \end{proof}
This expression of the stochastic process $X$ is actually not handy for the construction of the algorithm. We would like, as for Onstein-Uhlenbeck processes in \cite{Herrmann-Massin-1}, to transform the martingale part of the diffusion into a time-changed Brownian motion. However, we cannot apply such a transformation in the L-class framework, that is why we shall proceed in a quite different way.\\
To that end, let us suppose that $X$, solution of \eqref{edsl}, can be expressed using a time-changed Brownian motion:
\begin{equation}
X_t = f_L(t,x_0 + W_{\rho(t)}) ,\quad \forall t\ge 0,
\label{eqf}
\end{equation}
with $\rho(0) = 0$, $\rho'(t) >0$, for all $t \geqslant 0$ and $f_L(0,x)=x$ for any $x\in\mathbb{R}$.

\begin{lemma}
\label{frontL}
Let $\theta$ the function defined in \eqref{theta}. Then the  unique  weak solution of \eqref{edsl} is the process $(X_t,\ t\ge 0)$ defined in \eqref{eqf} with

\begin{align}
f_L(t,x) = \frac{\tilde{\sigma}(t)}{\sqrt{\rho'(t)}}x + c(t) &,\quad c(t) = e^{-\theta(t)}\int_0^t \beta(s) e^{\theta(s)}ds \nonumber\\
&\text{ and }  \rho(t) = \int^t_0 \tilde{\sigma}(s)^2e^{2\theta(s)} ds.
\label{defrho}
\end{align}
\end{lemma}

\begin{proof}
Let us first introduce the process $(M_t)_{t \in \mathbb{R}_+}$ defined by 
\begin{equation}
M_t := \int^t_0 \sqrt{\rho'(s)}dW_s
\end{equation}
where $W_t$ is the Brownian motion introduced in \eqref{edsl}. We notice that this process is a martingale with respect to the Brownian filtration and $\left\langle M\right\rangle_t = \int^t_0 \rho'(s)ds = \rho(t)$. We introduce the process $\hat{X}_t := f_L(t,x_0+ M_t)$. Using It\^o’s formula we get 
\begin{equation*}
d\hat{X}_t = \frac{\partial f_L}{\partial t}(t,M_t)dt + \frac{1}{2}\rho'(t)\frac{\partial^2 f_L}{\partial x^2}(t,M_t)dt + \frac{\partial f_L}{\partial x}(t,M_t)\sqrt{\rho'(t)}dW_t.
\end{equation*}
Computing all functions appearing in the previous equality, the stochastic process $\hat{X}_t$ is solution of \eqref{edsl}.
Using Dambis \& Dunbins-Schwarz Martingale representation theorem (see Theorem
V.1.6 p.170 \cite{rev}), there exists a Brownian motion $B_t$ such that 
\begin{equation}
M_t = B_{\left\langle M\right\rangle _t},\quad \forall t\ge 0.
\end{equation}
We deduce that $M_t \sim W_{\rho_L(t)}$ and therefore $(\hat{X}_t)_{t\ge 0} \sim (X_t)_{t\ge 0}$ with\\
 $X_t = f_L(t, x_0+W_{\rho(t)})$.
\end{proof}


\begin{rem}
\label{shiftL}
If the starting time associated to the study of the L-class diffusion is not the origin but another time $t_0$, then we also obtain an expression similar to \eqref{eqf}. Let $Y_t$ be the unique weak solution of

\[
\left \{
\begin{array}{c @{=} l}
    dY_t &  (\alpha(t+t_0)Y_t + \beta(t+ t_0))dt + \tilde{\sigma}(t+t_0)dW_t, \quad t\ge 0\\
     Y_0 & X_{t_0}. \\
\end{array}
\right.
\]
Then
\begin{equation}
Y_t = f_L(t+t_0,X_{t_0}e^{-\int^{t_0}_{0} \alpha(s) ds} + W_{\rho(t+t_0)-\rho(t_0)}) -e^{\int^{t + t_0}_{t_0} \alpha(s) ds}c(t_0),
\end{equation}
with $f_L$ given by Lemma \ref{frontL}.
\end{rem}

\subsection{Spheroids associated to a L-class diffusion process}
\subsubsection*{Introducing the exit time of the spheroid.}
We determine a specific spheroid for the diffusion by using the link with the time-changed Brownian motion. The boundaries of the spheroid associated to the diffusion starting at time $t_0$ in $x_0$ are denoted by $\psi_\pm^L(t;\,t_0,x_{0})$ and the corresponding exit time is
\begin{equation*}
\tau_L^{t_0} = \inf\{t >0:\ Y_t^L \notin [\psi_-^L(t;\,t_0,x_{0}),\psi_+^L(t;\,t_0,x_{0})]\}.
\end{equation*}

\begin{proposition}
\label{spheroidL}
Let us consider the spheroid starting in $(t_0, X_{t_0})$ with boundaries defined by
\begin{align*}
\psi_\pm^L(t;\,t_0,X_{t_0})& = 
e^{-\theta(t+t_0)}\,\psi_{\pm}(\rho(t+ t_0)-\rho(t_0)) + c(t+t_0)\\
&+ (X_{t_0}-c(t_0))e^{\int_{t_0}^{t+ t_0} \alpha(s) ds}
\end{align*}
for all $t\ge 0$, then the associated exit time satisfies
\begin{equation}
\tau_L^{t_0} \overset{d}{=} \rho^{-1}_L(\tau +\rho_L(t_0)) - t_0
\label{reltempsL}
\end{equation}
where $\tau =\inf\{ u > 0 \ : W_u \notin [\psi_-(t),\psi_+(t)]\}$, $\psi_\pm$ being defined in \eqref{sphbm}.
\end{proposition}

\begin{proof}
By definition,
\begin{align*}
&\tau_L^{t_0} = \inf\{t >0:Y_t \notin [\psi^L_-(t;\,t_0,X_{t_0}),\psi^L_+(t;\,t_0,X_{t_0})]\}\\
&=\inf\Big\{t >0:e^{-\theta(t+t_0)}\,W_{\rho(t+t_0)-\rho(t_0)} + c(t+t_0) + (X_{t_0} -c(t_0))e^{\int^{t + t_0}_{t_0} \alpha(s) ds}\\
&\quad \quad \quad \quad  \quad \quad \quad \quad \notin [\psi^L_-(t;\,t_0,X_{t_0}),\psi^L_+(t;\,t_0,X_{t_0})]\Big\}.
\end{align*}
Using $\psi^L_{\pm}$ introduced in the statement, we obtain the following expression for $\tau_L^{t_0}$:
\begin{align*}
&\inf\Big\{t >0:W_{\rho(t+t_0)-\rho(t_0)} \notin [\psi_{-}(\rho(t+ t_0)-\rho(t_0)),\psi_{+}(\rho(t+ t_0)-\rho(t_0))]\Big\}\\
&=\inf\{\rho^{-1}(u + \rho(t_0)) -t_0>0:\ W_u \notin [\psi_-(u),\psi_+(u)]\}\\
&=\rho_L^{-1}(\tau +\rho_L(t_0)) - t_0,
\end{align*}
where $\tau =\inf\{ u > 0:\ W_u \notin [\psi_-(u),\psi_+(u)]\}.$
\end{proof}

\subsubsection*{Size determination of the spheroids}

To define a WOMS algorithm for the L-class diffusions, we need to determine a suitable size for the spheroids in order to stay fully contained in the considered interval. Such size can be chosen by describing both the minimum and the maximum of the spheroid boundaries.\\
The size of the Brownian spheroid introduced in \eqref{sphbm} depends on a scaling parameter $d>0$, the support of the associated boundaries $\psi_\pm$ being therefore equal to $[0, d^2]$. Since the generalized spheroids used for L-class diffusion are directly linked to the Brownian ones, the parameter $d$ also changes their size and the boundaries $\psi_\pm^L$ are defined on the support $[0,\rho^{-1}(d^2 + \rho(t_0)) - t_0]$. Let us now precise this parameter $d$.
 
\begin{proposition}
\label{parameterL}
Let $m > 0$ and $0<\gamma<1$. For any $(x_0,t_0) \in [a,b]\times\mathbb{R}^+$ we define a parameter $d=d(x_0,t_0)$ such that the spheroid associated to the L-class diffusion starting in $(t_0,x_0)$ is totally included in $[a_{\gamma,x_0}, b_{\gamma,x_0}]$. Here $a_{\gamma,x_0}$ and $b_{\gamma,x_0}$ stands for $a_{\gamma,x} = a+ \gamma(x-a)$ and $b_{\gamma,x} = b- \gamma(b-x)$. This parameter is given by

\begin{equation}\label{paraml}
d=\left \{
\begin{array}{c @{\text{ if }} l}
    \frac{\min(1, \kappa_+)}{\Delta_m}(b_{\gamma,x_0}-x_0) &  b-x_0 \leqslant x_0-a\\[8pt]
    \frac{\min(1, \kappa_-)}{\Delta_m}(x_0-a_{\gamma,x_0})& x_0 -a \leqslant b-x_0
\end{array}
\right.
\end{equation}
where 
\begin{equation}
\Delta_m =e^{-\theta(t_0)}e^{\int_{t_0}^{t_0 +m} \vert\alpha(s)\vert ds}\left(\frac{1}{\sqrt{e}} + \sqrt{\int_{t_0}^{t_0+m} \frac{\vert\beta(s)+ x_0\,\alpha(s)\vert^2}{\tilde{\sigma}(s)^2}ds}\right),
\label{deltam}
\end{equation}
and $\kappa_\pm$ are defined by the following equations:
\begin{equation*}
 \kappa_+(b_{\gamma,x_0}-x_0) = \Delta_m \sqrt{\rho(t_0+m) -\rho(t_0)}\\
\end{equation*}
and
\begin{equation*}
  \kappa_-(x_0 -a_{\gamma,x_0}) =  \Delta_m \sqrt{\rho(t_0+m) -\rho(t_0)}.
\end{equation*}
\end{proposition}

\begin{rem}
\begin{itemize}
 \item The previous statement consists in finding $d$ such that
\[
\left \{
\begin{array}{c}
    d \leqslant \frac{1}{\Delta_m}(b_{\gamma,x_0}-x_0),\\[5pt]
    d \leqslant \frac{1}{\Delta_m}(x_0-a_{\gamma,x_0}),\\[5pt]
    d^2 \leqslant \rho(t_0+m) -\rho(t_0).
\end{array}
\right.
\]
The last condition in particular leads to $t\leqslant m$ since $\rho$ is a strictly increasing function.
\item It is possible to let $m$ depend on the couple $(t_0,x_0)$ which should permit to obtain bigger spheroids which are still included in the interval. Nevertheless for numerical purposes, such a procedure slows down drastically the algorithm we are going to present.
\item The choice of the constant $m$ is important, since it either slows down or speeds up the algorithm.
\item It is also possible to replace $x_0$ by $\max(\vert a\vert, \vert b \vert)$ in the definition of $\Delta_m$ which therefore becomes independent of the starting position $x_0$. Nevertheless such a replacement slows down the algorithm.
\end{itemize}
\end{rem}
\begin{proof}[Proof of Proposition \ref{parameterL}]
Let us first point out an upper bound for $\psi_+^L$ starting in $(t_0,x_0)$. We first require that  $d^2 \leqslant \rho(t_0+m) -\rho(t_0)$. Let us define $\mathcal{R}_+^L(t):=\psi_+^L(t;\,t_0,x_{0})-x_0$.  By definition
\begin{align*}
\mathcal{R}_+^L(t) &= e^{-\theta(t_0+t)}\left(\psi_+(\rho(t+t_0) -\rho(t_0)) + \int_{t_0}^{t_0+t} \beta(s) e^{-\int_0^s \alpha(u)du}ds\right)\\
&\quad +x_0\left(e^{\int_{t_0}^{t+t_0}\alpha(u)du}-1\right)
\end{align*}
Recalling \eqref{boundmb}, we obtain
\begin{align*}
\mathcal{R}_+^L(t) &\le e^{-\theta(t_0+t)}\left(\frac{d}{\sqrt{e}} + \int_{t_0}^{t_0+t} \beta(s) e^{\theta(s)}ds\right)+x_0\, e^{-\theta(t_0+t)} \left(e^{\theta(t_0)}-e^{\theta(t+t_0)}\right)
\end{align*}
\begin{align*}
\mathcal{R}_+^L(t) &\leqslant e^{-\theta(t_0+t)}\left(\frac{d}{\sqrt{e}} + \int_{t_0}^{t_0+t} \beta(s) e^{-\int_0^s \alpha(u)du}ds\right)\\
& +x_0 \,e^{-\theta(t_0+t)}\left(\int_{t_0}^{t+t_0}\alpha(s)e^{-\int_0^s \alpha(u)du}\right)\\
&\leqslant e^{-\theta(t_0)+\int_{t_0}^{t_0 +t}\vert\alpha(s)\vert ds}\left(\frac{d}{\sqrt{e}} + \int_{t_0}^{t_0+t} \frac{\vert\beta(s)+ x_0\,\alpha(s)\vert}{\tilde{\sigma}(s)}\,\tilde{\sigma}(s) e^{-\int_0^s \alpha(u)du}ds\right),
\end{align*}
since $\tilde{\sigma}$ is a positive function. Using Cauchy-Schwarz's inequality, we obtain the following upper-bound for $\mathcal{S}_+^L(t):=e^{\theta(t_0)}e^{-\int_{t_0}^{t_0 +t}\vert\alpha(s)\vert ds}\,\mathcal{R}_+^L(t)$:
\begin{align*}
\mathcal{S}_+^L(t)&\le \frac{d}{\sqrt{e}} + \left(\int_{t_0}^{t_0+t} \frac{\vert\beta(s)+ x_0\,\alpha(s)\vert^2}{\tilde{\sigma}(s)^2}\,ds\ \int_{t_0}^{t_0 + t}\tilde{\sigma}(s)^2 e^{-2\int_0^s \alpha(u)du}ds\right)^{1/2}\\
&= \frac{d}{\sqrt{e}} + \left(\int_{t_0}^{t_0+t} \frac{\vert\beta(s)+ x_0\,\alpha(s)\vert^2}{\tilde{\sigma}(s)^2}ds\right)^{1/2}\left(\rho(t+t_0) -\rho(t_0)\right)^{1/2}.
\end{align*}
Using $\rho(t_0+t) -\rho(t_0) \leqslant d^2$ and $t \leqslant m$, leads to
\begin{align*}
\mathcal{R}_+^L(t) &\leqslant de^{-\theta(t_0)+\int_{t_0}^{t_0 +m} \vert\alpha(s)\vert ds}\left(\frac{1}{\sqrt{e}} + \sqrt{\int_{t_0}^{t_0+m} \frac{\vert\beta(s)+ x_0\,\alpha(s)\vert^2}{\tilde{\sigma}(s)^2}ds}\right)\\
&= d \Delta_m.
\end{align*}
Under the condition $d \Delta_m + x_0 \leqslant b_{\gamma,x_0}$, we observe that the spheroid belongs to the interval  $d \Delta_m + x_0 \leqslant b_{\gamma,x_0}$. Therefore we shall choose
\begin{equation}
d \leqslant \frac{1}{\Delta_m}(b_{\gamma,x_0}-x_0).
\label{magel}
\end{equation}
Let us now deal similarly with a lower-bound of $\psi_-^L$. We define \[\mathcal{R}_-^L(t):=\psi_-^L(t;\,t_0,x_{0}) -x_0.\] Hence
\begin{align*}
\mathcal{R}_-^L(t) &= e^{-\theta(t_0+t)}\left(\psi_-(\rho(t+t_0) -\rho(t_0)) + \int_{t_0}^{t_0+t} \beta(s) e^{-\int_0^s \alpha(u)du}ds\right)\\
&\quad +x_0\left(e^{\int_{t_0}^{t+t_0}\alpha(u)du}-1\right)\\
&\geqslant e^{-\theta(t_0+t)}\left(-\frac{d}{\sqrt{e}} + \int_{t_0}^{t_0+t} (\beta(s)+ x_0\,\alpha(s))\ e^{-\int_0^s \alpha(u)du}ds\right)\\
&\geqslant e^{-\theta(t_0+t)}\left(-\frac{d}{\sqrt{e}} - \int_{t_0}^{t_0+t} \vert\beta(s)+ x_0\,\alpha(s)\vert e^{-\int_0^s \alpha(u)du}ds\right)\\
&\geqslant -e^{-\theta(t_0)}e^{\int_{t_0}^{t_0 +m} \vert\alpha(s)\vert ds}\left(\frac{d}{\sqrt{e}} + \int_{t_0}^{t_0+t} \vert\beta(s)+ x_0\,\alpha(s)\vert e^{-\int_0^s \alpha(u)du}ds\right).
\end{align*}
Using then the same arguments as for the upper bound, we obtain
\begin{equation*}
\psi_-^L(t;\,t_0,x_{0}) \geqslant - \Delta_m d + x_0.
\end{equation*}
The condition $- \Delta_m d + x_0 \geqslant a_{\gamma,x_0}$ is equivalent to
\begin{equation}
d \leqslant \frac{1}{\Delta_m}(x_0 -a_{\gamma,x_0}).
\label{minl}
\end{equation}
Combining \eqref{magel}, \eqref{minl} and $d^2 \leqslant \rho(t_0+m) -\rho(t_0)$, we deduce the announced statement. 
\end{proof}

\subsection{WOMS algorithm for L-class diffusions}
Let us present now the random walk on spheroids which permits to approximate the L-class diffusion exit time.

\centerline{{\sc Algorithm}$_{m}$ (L-class WOMS)}

\noindent\fbox{\begin{minipage}{0.95\textwidth}
Parameter: $m>0$.

\textit{Initialization:} $Z_0 = x_0$ and $\mathcal{T}_\epsilon =0$

\textit{From step $n$ to step $n+1$:}

While $Z_n \leqslant b - \epsilon$ and $Z_n \geqslant a + \epsilon$ do
\begin{itemize}
\item[$\bullet$] simulate a Brownian exit time from the spheroid defined by $\psi_\pm$ with coefficient  $d=d(T_n,Z_n)$ defined in \eqref{paraml}.\\ We denote this random time $\tau_{n+1}$.

\item[$\bullet$] set $\tau^L_{n+1}=\rho^{-1}(\tau_{n+1}+\rho(\mathcal{T}_\epsilon))-\mathcal{T}_\epsilon$.

\item[$\bullet$] simulate a Bernoulli distributed r.v. $\mathcal{B}\sim\mathcal{B}(\frac{1}{2})$, if $\mathcal{B}=1$ then set

 $Z_{n+1}=\psi^L_-(\tau^L_{n+1};\,\mathcal{T}_\epsilon, Z_n)$ otherwise set $Z_{n+1}=\psi^L_+(\tau^L_{n+1};\,\mathcal{T}_\epsilon, Z_n)$.

\item[$\bullet$] $\mathcal{T}_\epsilon\leftarrow \mathcal{T}_\epsilon +\tau^L_{n+1}$.
\end{itemize}
\textit{Outcome:} $\mathcal{T}_\epsilon$ the approximated exit time from the interval $[a,b]$ for the diffusion $(X_t,\,t\ge 0)$.
\end{minipage}}

\vspace*{0.2cm}
As usual let us describe the efficiency of the algorithm. This algorithm is particularly efficient since its averaged  number of steps is of the order $|\log(\epsilon)|$ and since its outcome $\mathcal{T}_\epsilon$ converges towards the value of the exit time as $\epsilon$ tends to $0$. We present these two results in details in the following subsections. Even if the statement of these results look like similar to those presented in the Ornstein-Uhlenbeck context (see \cite{Herrmann-Massin-1}), the situations are clearly different since here the coefficients - and therefore the size of the spheroids - are time-dependent.\\
Since the L-class diffusions are non homogeneous, the sequence $(Z_n)_n$ of successive exit positions, appearing in the algorithm,  does not define a Markov chain. We need therefore to consider both the successive times and positions $(T_n,X_n)$ in order to deal with a Markov chain. Here $T_n$ stands for the cumulative time:
\begin{equation}
T_n=\sum_{k=1}^n\tau^L_k,\quad n\ge 1.
\label{defTn}
\end{equation}

\subsubsection{Average number of steps}

In order to describe precisely the average number of steps in {\sc Algorithm}$_{m}$, we introduce two crucial hypotheses.

\begin{assumption}
\label{assump1}
There exist $q'\in[0,1[$ and $q\in[0,1]$, $C_{\tilde{\sigma},\beta}>0$ and $\underline{\sigma}>0$ such that 

\begin{equation}
\vert\alpha(t)\vert =\mathcal{O}((\ln t)^{q'}) ,\ \quad \mbox{for large values of}\ t,
\end{equation}
and
\begin{equation}
\label{minsigma}
\underline{\sigma}\leqslant\tilde{\sigma}(t) \leqslant C_{\tilde{\sigma}, \beta}\ t^{q/4}, \quad \vert \beta(t)\vert \leqslant  C_{\tilde{\sigma}, \beta}\ t^{q/4}, \quad\mbox{for}\ t \ \mbox{large enough}.
\end{equation}
\end{assumption}
\begin{assumption}
\label{assump2}
There exists $\chi_m > 0$ such that, for any $t$ large enough,
\begin{equation}
\inf\limits_{s \in [t,t+m]} \tilde{\sigma}(s) \geqslant \chi_m \sup\limits_{s \in [t,t+m]} \tilde{\sigma}(s).
\end{equation}
\end{assumption}

\begin{thm}
\label{meanL}
Let us assume that Assumptions \ref{assump1} and \ref{assump2} are satisfied for a particular parameter $m>0$. Then for any parameter $\tilde{q}>q$, there exists a constant $C_{\tilde{q}}>0$ such that $N_\epsilon$, the number of steps observed in $\mbox{{\sc Algorithm}}_{m}$ has the following upper-bound:
\begin{equation*}
\mathbb{E}[N_{\epsilon}^{1-\tilde{q}}] \leqslant C_{\tilde{q}}|\log(\epsilon)|,
\end{equation*}
for any $\epsilon>0$ small enough.
\end{thm}

In particular, for a L-class diffusion with bounded coefficients, we  can prove that
\(\mathbb{E}[N_{\epsilon}] \leqslant C_{0}|\log(\epsilon)|,\)
for $\epsilon$ small enough.\\

Let us notice that {\sc Algorithm}$_m$ can be modified in order to approximate the stopping time $\mathcal{T}\wedge T_{\rm max}$ where $T_{\rm max}$ is a fixed time horizon. It suffices in such a situation to observe the path skeleton $(T_n,X_n)_{n\ge 0}$ up to the exit from the domain $[0,T_{\rm max}]\times[a+\epsilon, b-\epsilon]$. The proof of Theorem \ref{meanL} can be adapted to this modified algorithm: there exists a constant $C>0$ such that the average number of spheroids satisfies
\[
\mathbb{E}[N_{\epsilon}] \leqslant C|\log(\epsilon)|,
\] 
for any $\epsilon>0$ small enough. Since this result only concerns the diffusion process on the restricted time interval $[0,T_{\rm max}]$, we don't need any particular assumption on the large time behaviour of the coefficients $\alpha$, $\beta$ and $\tilde{\sigma}$. Assumption \ref{assump1} and \ref{assump2} are therefore not necessary for the modified algorithm.\\

We postpone the proof of Theorem \ref{meanL} and present several preliminary results. First we shall focus our attention on a comparison result between the L-class diffusion and a particular autonomous diffusion. Secondly we describe particular solutions of PDEs related to the diffusion generator. Finally we prove Theorem \ref{meanL} using the martingale theory.

\subsubsection*{A comparison result for SDEs}

We introduce two different results: the first one permits to skip the diffusion coefficient in \eqref{edsl} and the second one permits to replace the time-dependent  drift term by a constant drift.

\begin{proposition}
Let $(X_t,\ t\ge 0)$ the solution of the SDE \eqref{edsl}. We define the strictly increasing function $\gamma$ by
\[
\int_0^{\gamma(t)}\tilde{\sigma}^2(s)ds=t,\quad t\ge 0.
\]
Then  $Y_t:=X_{\gamma(t)}$ satisfies the following SDE

\begin{equation}
\label{eq:eds}
dY_t=\Big( \frac{\alpha(\gamma(t))}{\tilde{\sigma}^2(\gamma(t))}\, Y_t+  \frac{\beta(\gamma(t))}{\tilde{\sigma}^2(\gamma(t))}\Big)\,dt+dB_t,\quad t\ge 0.
\end{equation}
where $(B_t)_{t \ge 0}$ is a one-dimensional Brownian motion.
\end{proposition}

\begin{proof}

Using the definition of $Y_t$, we get
\begin{align*}
Y_t&=X_{\gamma(t)}=x+\int_0^{\gamma(t)}\Big( \alpha(s)X_s+\beta(s) \Big)\,ds+\int_0^{\gamma(t)}\tilde{\sigma}(s)dW_s\\
&=x+\int_0^t\Big( \alpha(\gamma(s))X_{\gamma(s)}+\beta(\gamma(s)) \Big)\gamma'(s)\,ds+B_{t}\\
&=x+\int_0^t\Big( \alpha(\gamma(s))Y_s+\beta(\gamma(s)) \Big)\gamma'(s)\,ds+B_{t}
\end{align*}
where $B_t=\int_0^{\gamma(t)}\tilde{\sigma}(s)dW_s$ is a standard Brownian motion.
\end{proof}
We obtain the following comparison result, its proof can be found in \cite{ikeda-wanatabe} (Chapter VI).
\begin{proposition}
\label{comparison}
Let $T>0$ and let us define 

\[
\mu_T:=\inf_{x\in[a,b],\,t\le \gamma^{-1}(T)}\Big\{ \frac{\alpha(\gamma(t))}{\tilde{\sigma}^2(\gamma(t))}\, x+  \frac{\beta(\gamma(t))}{\tilde{\sigma}^2(\gamma(t))}\Big\}.
\]
Let $(Z^T_t)_{t \ge 0}$ the Brownian motion with drift satisfying
\begin{equation}
Z_t^T=x +\mu_T\,t+B_t,\quad t\ge 0.
\label{zTt}
\end{equation}
Then $(Y_t)$ the solution of \eqref{eq:eds} with initial condition $x$ satisfies
\begin{equation*}
(Z_t^T\le Y_t\quad\mbox{a.s.},\quad \forall t\le \gamma^{-1}(T))\quad\mbox{and}\quad (Z_{\gamma(t)}^T\le X_t\quad\mbox{a.s.}\quad\forall t\le T). 
\end{equation*}
\end{proposition}

\begin{rem}
\label{inversecomparison}
Choosing rather the particular value
\[
\mu_T:=\sup_{x\in[a,b],\,t\le \gamma^{-1}(T)}\Big\{ \frac{\alpha(\gamma(t))}{\tilde{\sigma}^2(\gamma(t))}\, x+  \frac{\beta(\gamma(t))}{\tilde{\sigma}^2(\gamma(t))}\Big\},
\]
leads to ($Z_t^T\ge Y_t$ a.s. for all $t\leqslant \gamma^{-1}(T)$).
\end{rem}

\subsubsection*{An Initial-Boundary Value problem}

We consider a value problem which is directly linked to the L-class diffusions: let $F:(\mathbb{R}_+,[a,b])\to \mathbb{R}$ be the solution of
\begin{equation}
\frac{\partial F}{\partial t } + (\alpha(t)x + \beta(t))\frac{\partial F}{\partial x} + \frac{1}{2}\tilde{\sigma}(t)^2\frac{\partial^2 F}{\partial x^2} = 0
\label{edp}
\end{equation}
 with initial and boundary conditions  $F(0,x) = x$, $F(t,a) = a$, $F(t,b) = b$.

It is well-known (see, for instance, \cite{bass}, Chap.II) that $F$ admits a probabilistic representation. Indeed 

\begin{equation}
F(t,x) = \mathbb{E}_x[X_{t\wedge \tau_{ab}}],\quad \forall t\ge 0,\quad \forall x\in [a,b],
\label{defF}
\end{equation}
where $(X_t,\,t\ge 0)$ satisfies \eqref{edsl}.
We list some useful properties of the function $F$.

\begin{proposition}
The function $x\mapsto F(t,x)$ defined in \eqref{defF} is increasing on the set $[a,b]$.
\end{proposition}

\begin{proof}
It suffices to compare two paths $X$ and $X'$, having different starting points $x$ and $x'$ with $x\geqslant x'$ and satisfying the same SDE. By coupling properties, we obtain that for all $s\geqslant 0$, $X_s \geqslant X'_s$ and if there exists $s_0$ such that $X_{s_0} = X'_{s_0}$ then $X_{s} = X'_{s}$ for all $s\geqslant s_0$. Several cases can occur concerning the values of $X_{t\wedge \tau_{ab}}$ and $X'_{t\wedge \tau_{ab}}$. Either both exit times occur after the fixed time $t$, either both exit times occur before $t$, either only one of them occurs before $t$. Different situations are illustrated in Figure \ref{figure}. Observing carefully all possible scenarios, it is straightforward to observe $X_{t\wedge \tau_{ab}}\ge X'_{t\wedge \tau_{ab}}$ in any case.
\begin{figure}[H]
\begin{minipage}{0.49\textwidth}
\begin{figure}[H]
\includegraphics[width=\textwidth]{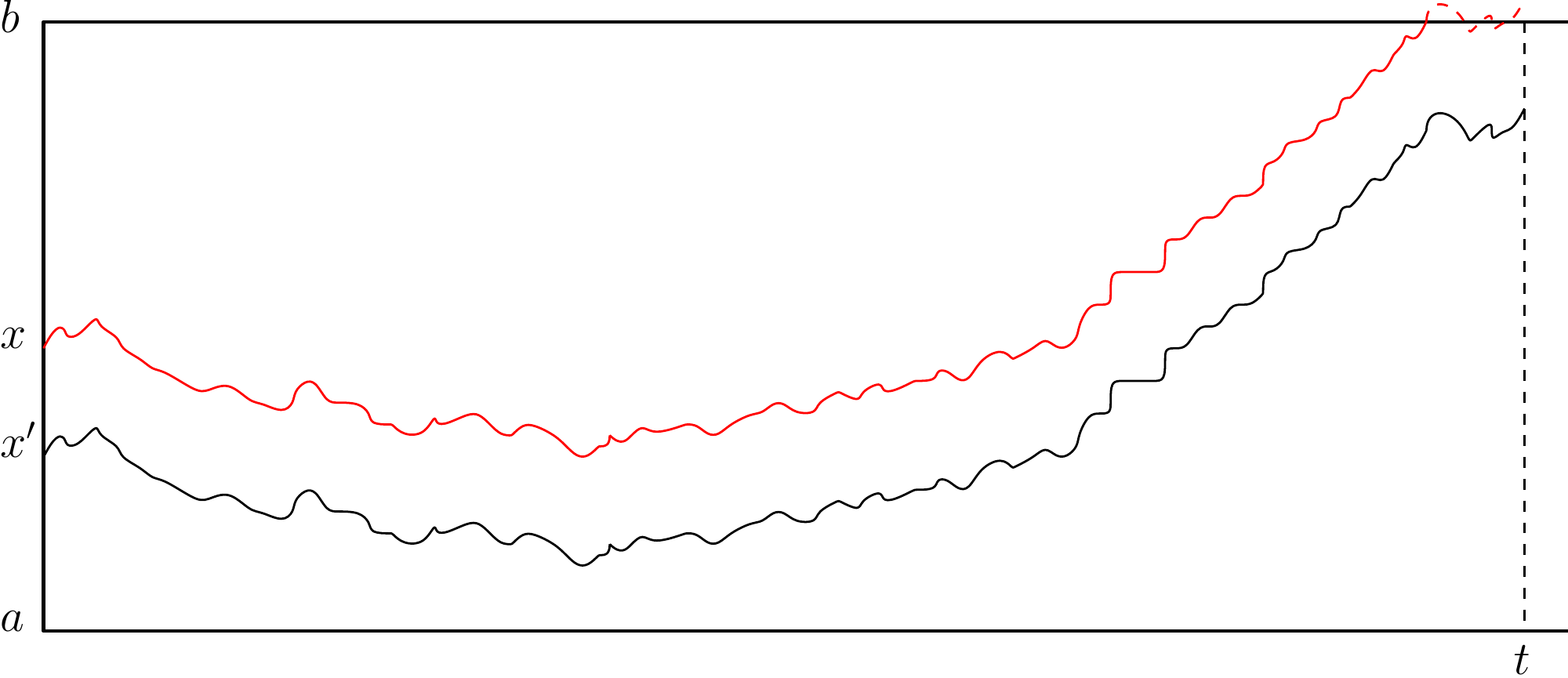}
\end{figure}
\end{minipage}
\begin{minipage}{0.49\textwidth}
\begin{figure}[H]
\includegraphics[width=\textwidth]{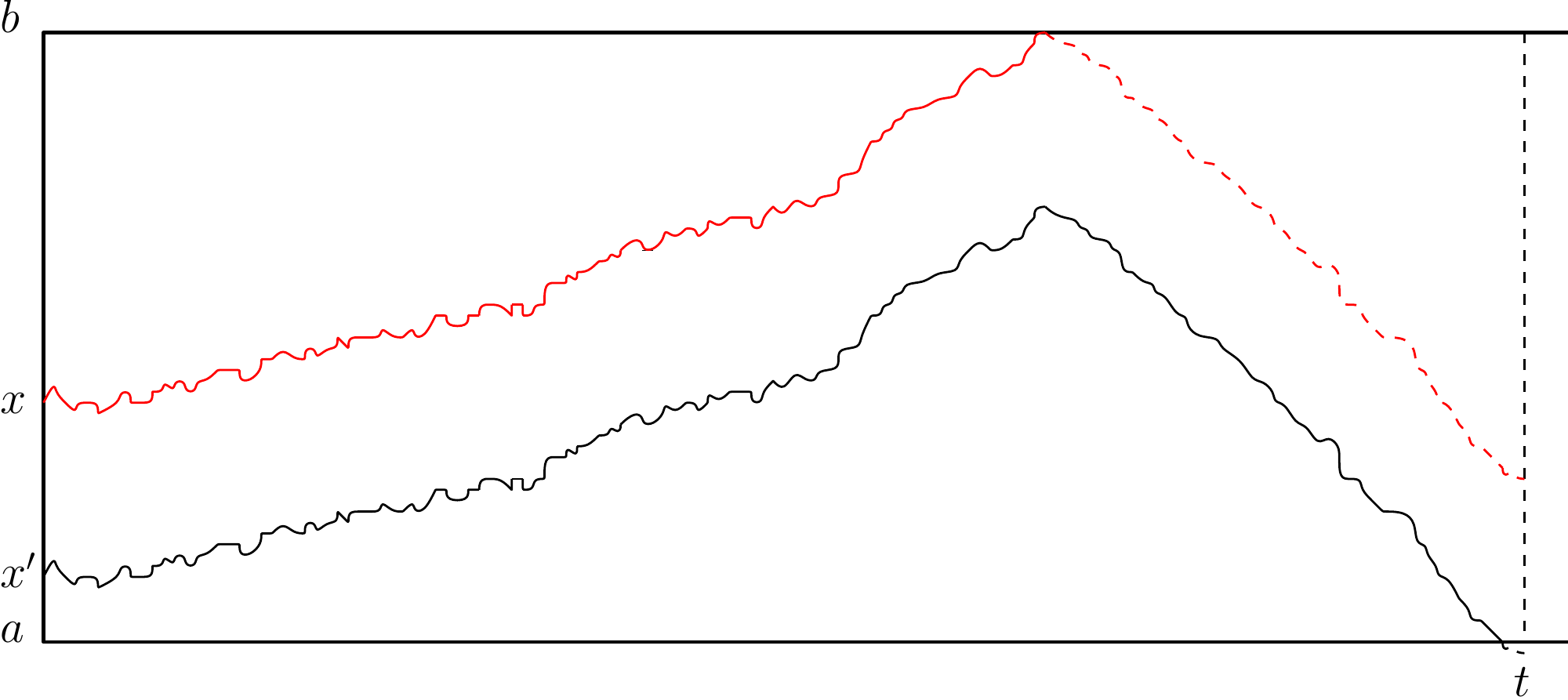}
\end{figure}
\end{minipage}

\begin{minipage}{0.49\textwidth}
\begin{figure}[H]
\includegraphics[width=\textwidth]{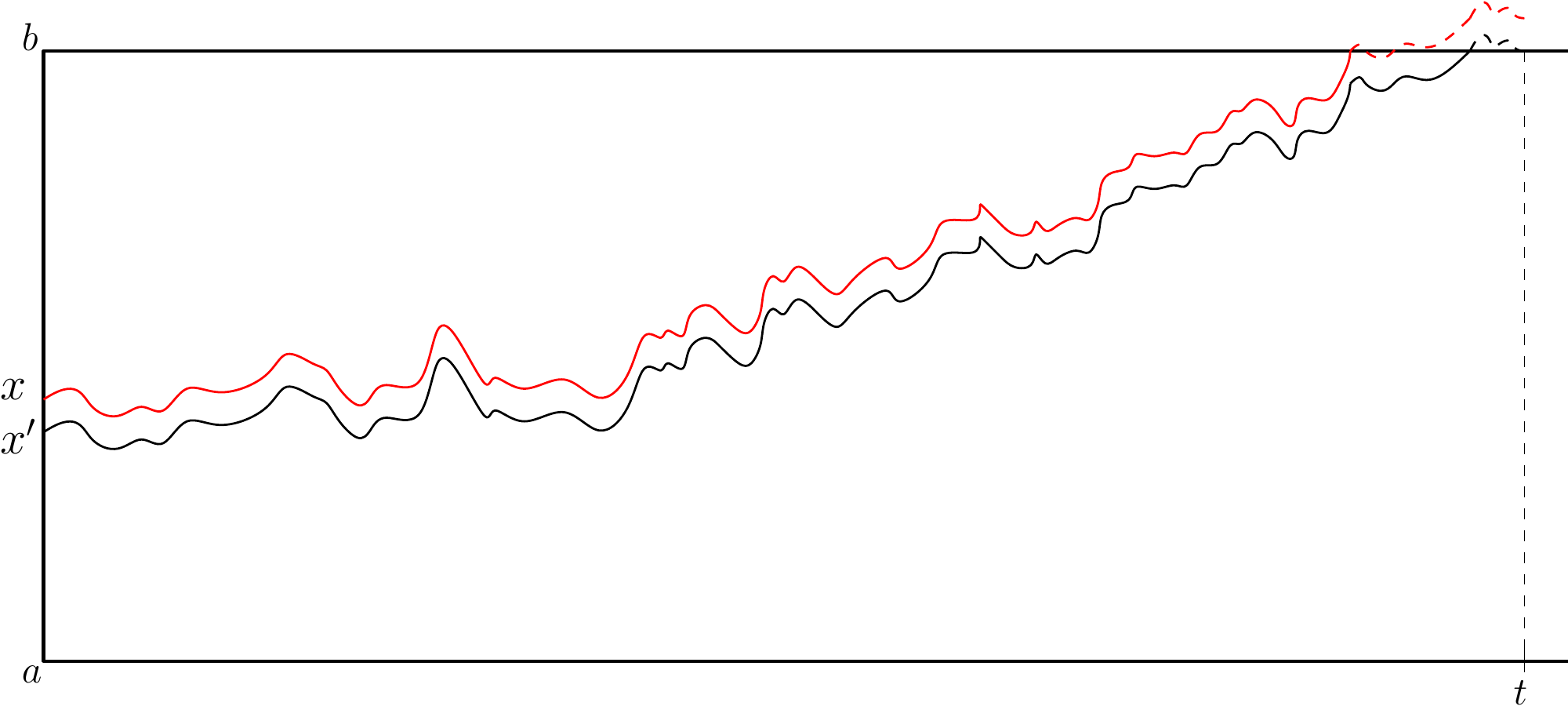}
\end{figure}
\end{minipage}
\begin{minipage}{0.49\textwidth}
\begin{figure}[H]
\includegraphics[width=\textwidth]{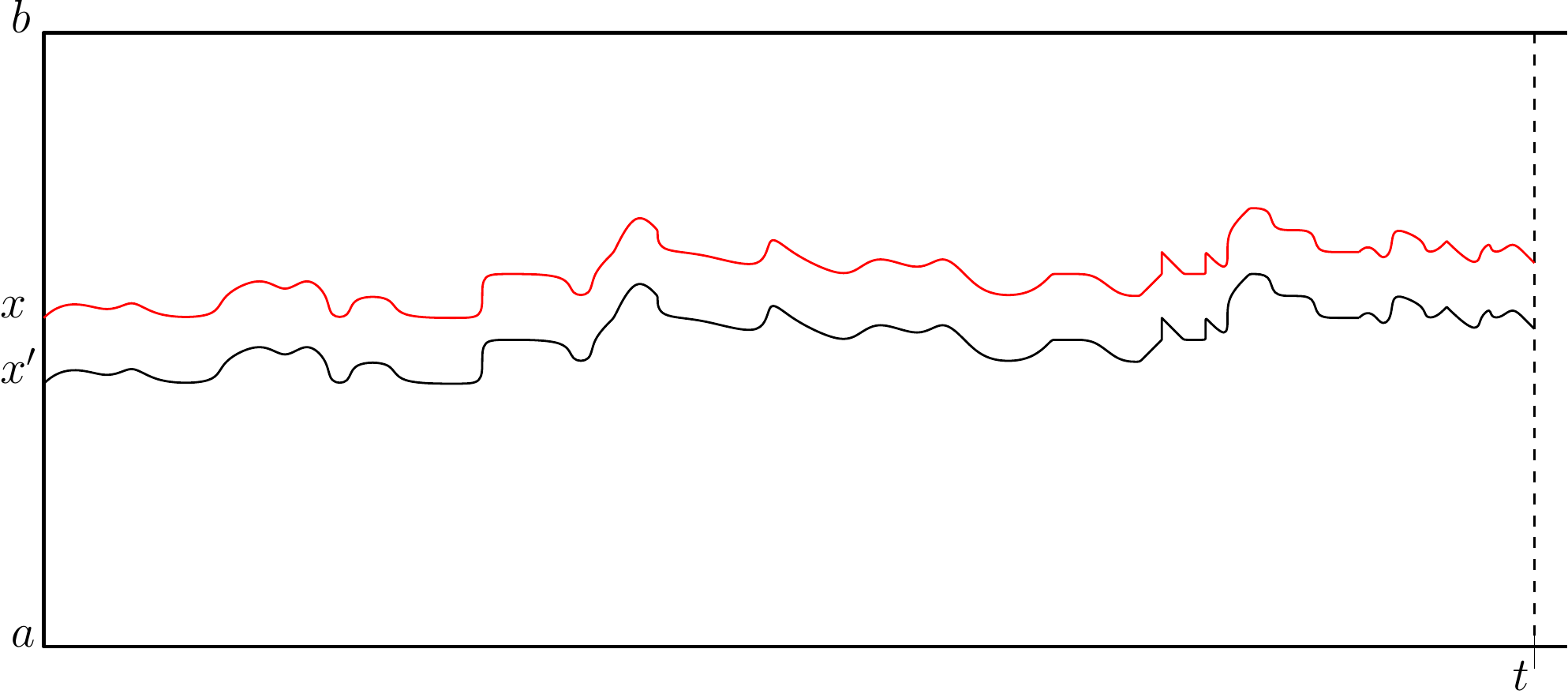}
\end{figure}
\end{minipage}
\caption{\small Possible scenarios occuring when observing the two different paths driven by the same noise.}
\label{figure}
\end{figure}
\end{proof}

\begin{proposition}
\label{continuity}
The function $x\mapsto F(t,x)$ defined in \eqref{defF} is continuous on the interval $[a,b]$. 
\end{proposition}

\begin{proof}
We consider two strong solutions $(X_t^x)_{t \ge 0}$ and $(\tilde{X}_t^{x+h})_{t \ge 0}$ satisfying \eqref{edsl} with different starting point spaced by $h>0$. The exit time of the diffusion $X^x$ (respectively $\tilde{X}^{x+h}$) should be denoted by $\tau_{ab}$ (resp. $\tilde{\tau}_{ab}$) but for notational simplicity, we skip the index $ab$. Let $T$ be a fixed time, using the definition of $F$, we have 
\begin{align*}
0 &\leqslant F(T,x+h) - F(T,x)\\
 &= \mathbb{E}[\tilde{X}^{x+h}_{T \wedge \tilde{\tau}} -\tilde{X}^{x+h}_{T \wedge \tau \wedge \tilde{\tau}} + \tilde{X}^{x+h}_{T \wedge \tau \wedge \tilde{\tau}} -X^{x}_{T \wedge \tau} + X^{x}_{T \wedge \tau \wedge \tilde{\tau}} -X^{x}_{T \wedge \tau \wedge \tilde{\tau}}]\\
&= \mathbb{E}[ \chi_{T \wedge \tau \wedge \tilde{\tau}}] +\mathbb{E}[(\tilde{X}^{x+h}_{T \wedge \tilde{\tau}} -\tilde{X}^{x+h}_{T \wedge \tau})1_{\{\tilde{\tau} \geqslant \tau\}} -(X^{x}_{T \wedge \tau} -X^{x}_{T \wedge \tilde{\tau}})1_{\{\tilde{\tau} \leqslant \tau\}}],
\end{align*}
where  $\chi_{T \wedge \tau \wedge \tilde{\tau}} := \tilde{X}^{x+h}_{T \wedge \tau \wedge \tilde{\tau}} - X^x_{T \wedge \tau \wedge \tilde{\tau}} = he^{\int_0^{T \wedge \tau \wedge \tilde{\tau}} \alpha(u)du} \leqslant he^{\int_0^{T} \alpha(u)du}$ for all $T\ge 0$.
Let $\delta > 0$.
We can split each term as follows
\begin{align*}
\mathbb{E}[(\tilde{X}^{x+h}_{T \wedge \tilde{\tau}} -\tilde{X}^{x+h}_{T \wedge \tau})1_{\{\tilde{\tau} \geqslant \tau\}}] &\leqslant  \mathbb{E}[(\tilde{X}^{x+h}_{T \wedge \tilde{\tau}} -\tilde{X}^{x+h}_{T \wedge \tau})1_{\{\tilde{\tau} \geqslant \tau,\  \tilde{X}^{x+h}_{T \wedge \tilde{\tau}} -\tilde{X}^{x+h}_{T \wedge\tau} > \delta\}}]\\
&+ \delta\mathbb{P}(\tilde{\tau} \geqslant \tau, 0 \leqslant\tilde{X}^{x+h}_{T \wedge \tilde{\tau}} -\tilde{X}^{x+h}_{T \wedge\tau} \leqslant \delta)\\
&\leqslant (b-a)\mathbb{P}(\tilde{\tau} \geqslant \tau, \tilde{X}^{x+h}_{T \wedge \tilde{\tau}} -\tilde{X}^{x+h}_{T \wedge\tau} > \delta)+ \delta\\
&\leqslant (b-a)\mathbb{P}(\tilde{\tau} \geqslant \tau,\ T\geqslant\tau, \tilde{X}^{x+h}_{T \wedge \tilde{\tau}} -\tilde{X}^{x+h}_{\tau} > \delta)+ \delta.
\end{align*}
Similarly we obtain for the second term:
\begin{align*}
\mathbb{E}[(X^{x}_{T \wedge\tilde{\tau}} - X^{x}_{T \wedge \tau})1_{\tilde{\tau} \leqslant \tau}] 
&\leqslant (b-a)\mathbb{P}(\tilde{\tau} \leqslant \tau, T\geqslant \tilde{\tau}, X^{x}_{ \tilde{\tau}} -X^{x}_{t \wedge\tau} > \delta)+ \delta.
\end{align*}
Both probabilities appearing in the previous upper-bound can be treated in a similar way. We develop the arguments just for one of them:  $\mathbb{P}(\tilde{\tau} \geqslant \tau, t\geqslant \tau, \tilde{X}^{x+h}_{t \wedge \tilde{\tau}} -\tilde{X}^{x+h}_{\tau} > \delta)$. Let us introduce the shift process $\xi_t=\tilde{X}_{\tau+t}^{x+h}$. If $\tau$ is known (let us say that it is equal to $\phi$) then, due to the Markov property of the diffusion, $(\xi_t)_{t \geqslant 0}$ satisfies the following SDE:
\begin{equation}
d\xi_t=(\alpha(t+\phi)\xi_t+\beta(t+\phi))\,dt+\tilde{\sigma}(t+\phi)dB_{t},
\label{ksieds}
\end{equation}
where $(B_t)_{t \ge 0}$ is a standard Brownian motion and $\xi_0=\tilde{X}_{\tau}^{x+h}$. Since $\tilde{X}^{x+h}$ and $X^x$ are two strong solutions and since $h>0$, we have $\tilde{X}^{x+h}_t\ge X^x_t$ for any $t\ge 0$. In particular, the event $\tau\le \tilde{\tau}$ implies that $X^x_{\tau}=a$. Therefore on the event  $\tau\le \tilde{\tau}$,
\[
\xi_0\le a+h\exp\int_0^\phi|\alpha(u)|\,du=:a+h\Theta(\phi).
\]
By applying the comparison result described in Proposition \ref{comparison} (just replacing $\alpha$ by $\alpha(\cdot+\phi)$, $\beta$ by $\beta(\cdot+\phi)$ and $\tilde{\sigma}$ by $\tilde{\sigma}(\cdot+\phi)$) we obtain that $\xi_{\gamma(t)}\ge Z_t^T$ defined in \eqref{zTt} for all $t\le\gamma^{-1}(T)$. Of course $Z^T$ and $\xi$ have the same initial condition. If we denote by $\mathcal{T}_l$ the first passage time through the level $l$, then 
 \begin{equation}
 \mathbb{P}(\tilde{\tau} \geqslant \tau,T\geqslant \tau, \tilde{X}_{T\wedge\tilde{\tau}}^{x+h} -\tilde{X}^{x+h}_{\tau} > \delta) \leqslant \mathbb{P}_{a+h \Theta(T)}(\mathcal{T}_{a+\delta+h \Theta(T)}(Z^T) \leqslant \mathcal{T}_a(Z^T)),
 \label{majprobaTheta}
 \end{equation}
 since $\Theta(\phi)\le \Theta(T)$.\\
Let us now let $\delta$ depend on $h$, namely $\delta=\sqrt{h}$. Using the scale function of a drifted Brownian motion, we obtain in the small $h$ limit:
\begin{align*}
&\mathbb{P}_{a+h\Theta(T)}\Big(\mathcal{T}_{a+\delta+ h\Theta(T)}(Z^T) \leqslant \mathcal{T}_a(Z^T)\Big) = \frac{e^{-2\mu_T (a +h\Theta(T))}- e^{-2\mu_T a}}{e^{-2\mu_T(a+\delta)} - e^{-2 \mu_T a}}\\
&\sim -h\Theta(T)\frac{2\mu_T e^{-2 \mu_T a}}{e^{-2\mu_T (a+\delta)} - e^{-2 \mu_T a}}\sim -h\Theta(T)\frac{2\mu_T e^{-2 \mu_T a}}{\sqrt{h}} \\
&= -2\sqrt{h}\Theta(T) \mu_T e^{-2 \mu_T a}.
\end{align*}
Finally we observe that  $F(T,x+h)$ converges towards $F(T,x)$ as $h$ tends to $0_+$. By symmetry we obtain also the result for $h\to 0_-$.
\end{proof}

	

 


\begin{proposition}
\label{minderiv}
There exists $\kappa > 0$ such that for all $(t,x) \in \mathbb{R}_+ \times[a,b]$, $\frac{\partial F}{\partial x}(t,x) \geqslant \kappa.$
\end{proposition}

\begin{proof}
First let us recall that $F$ has a probabilistic representation given by \eqref{defF}. We shall use this representation in order to lower bound the space derivative. We consider two different cases: small times, that is $t\le 2$, or large times $t>2$.\\
\underline{First case: $t> 2$.}\\
We denote by $\tau^x$ the first time the process $X^x$ starting at $x$ exits from the interval $]a,b[$ and by $\tau_-^x$ (respectively $\tau_+^x$) the first exit time from $]a,b_h[$ (resp. from the first exit time from $]a_h,b[$) with
\begin{equation}
b_h:=b-he^{\int_0^1|\alpha(s)|\,ds}\quad\mbox{and}\quad a_h:=a+he^{\int_0^1|\alpha(s)|\,ds}.
\label{ahbh}
\end{equation}
We also introduce $(Y^\pm_t)$ the solutions of the shifted SDEs:
\begin{equation}
\label{eq:edsmark}
dY^-_t=(\alpha(t+\tau^x_-)Y_t+\beta(t+\tau^x_-))\,dt+\tilde{\sigma}(t+\tau^x_-)dW_{t+\tau^x_-}, 
\end{equation}
with the initial condition $Y^-_0=a+he^{-\int_0^1|\alpha(s)|\,ds}$ and 
\begin{equation}
\label{eq:edsmark2}
dY^+_t=(\alpha(t+\tau^{x+h}_+)Y_t+\beta(t+\tau^{x+h}_+))\,dt+\tilde{\sigma}(t+\tau^{x+h}_+)dW_{t+\tau^{x+h}_+}, 
\end{equation}
with the initial condition $Y^+_0=b-he^{-\int_0^1|\alpha(s)|\,ds}$. We associate  the stopping times $\mathcal{T}(Y^\pm)$, the exit time from $]a,b[$ and $\mathcal{T}_a(Y^\pm)$ (resp. $\mathcal{T}_b(Y^\pm)$) the first passage times through levels $a$ and $b$ to these diffusions. 

In order to minimize the derivative of $F$, we need to lower bound the following expectation, for $h>0$:
\[
F(t,x+h)-F(t,x)=\mathbb{E}[X^{x+h}_{\tau^{x+h}\wedge t}-X^{x}_{\tau^{x}\wedge t}].
\]
Let us observe particular scenarios which permit the difference between the diffusions to be equal to the maximal value $b-a$. To that end, we introduce two events:

\[
E_{ab}:=\{ \tau^x_-\le 1, X^x_{\tau^x_-}=a,\, \mathcal{T}(Y^-)\le 1, Y^-_{\mathcal{T}(Y^-)}=b \},
\] \[
E_{ba}:=\{ \tau^{x+h}_+\le 1, X^{x+h}_{\tau^{x+h}_+}=b,\, \mathcal{T}(Y^+)\le 1, Y^+_{\mathcal{T}(Y^+)}=a \}.
\]
By Lemma \ref{lemmappendix} and Lemma \ref{lemmappendixbis} (presented in the Appendix)  $E_{ab}\cap E_{ba}=\emptyset$ and $E_{ab}\cup E_{ba}\subset \{X^{x+h}_{\tau^{x+h}\wedge t}- X^x_{\tau^x\wedge t}=b-a \}$ for all $t\ge 2$. Hence
\begin{equation}
F(t,x+h)-F(t,x)\ge (b-a)(\mathbb{P}( E_{ab})+\mathbb{P}( E_{ba})).
\label{minF}
\end{equation}
Let us first deal with $\mathbb{P}( E_{ab})$. Conditionally to $\tau^x_-=\phi$, the strong Markov property of the diffusion process implies that $Y_t^-$ has the same distribution as the solution of the SDE :
\begin{equation}
\label{eq:edsmark3}
d\xi_t=(\alpha(t+\phi)\xi_t+\beta(t+\phi))\,dt+\tilde{\sigma}(t+\phi)dB_{t}, \quad \xi_0=Y_0^-,
\end{equation}
where $(B_t)$ is a standard Brownian motion.  Since $\phi\le 1$ and $\mathcal{T}(Y^-)\le 1$ on the event $E_{ab}$, we need to describe the paths of the initial diffusions $X^x$ and $X^{x+h}$ on a time interval of length at most equal to $2$. We can easily adapt the comparison result of Proposition \ref{continuity} to obtain that $\xi_t\ge Z_{\gamma(t)}^T$ for all $t\le 1$ and $T=2$ ($Z_t^T$ being defined in the statement of Proposition \ref{comparison}). Let us notice that $\gamma$ here depends on $\phi$.
We deduce that
\begin{align}\label{min}
\mathbb{P}\Big(\mathcal{T}(Y^-)\le 1, Y^-_{\mathcal{T}(Y^-)}\Big)
&= \mathbb{P}\Big(\mathcal{T}_b(Z)\le \gamma^{-1}(1), Z_{\mathcal{T}(Z)}=b\Big|\tau^x_-=\phi\Big)\nonumber \\
&\ge \mathbb{P}\Big(\mathcal{T}_b(Z)\le \underline{\sigma}^2, Z_{\mathcal{T}(Z)}=b\Big|\tau^x_-=\phi\Big),
\end{align}
where $\underline{\sigma}$ is the uniform lower bound of $\tilde{\sigma}(t)$.  Indeed 
\[
\gamma^{-1}(1)=\int_0^1\tilde{\sigma}^2(s+\phi)\,ds\ge \underline{\sigma}^2.
\]
We observe that the lower bound in \eqref{min} does not depend on $\phi$. Consequently
\begin{align*}
\mathbb{P}(E_{ab})&=\mathbb{E}\Big[ 1_{\{ \tau^x_-\le 1, X^x_{\tau^x_-}=a \}}\mathbb{P}\Big(\mathcal{T}(Y^-)\le 1, Y^-_{\mathcal{T}(Y^-)}=b\Big|\tau^x_-\Big) \Big]\\
&\ge \mathbb{P}\Big(  \tau^x_-\le 1, X^x_{\tau^x_-}=a\Big)\mathbb{P}\Big(\mathcal{T}_b(Z)\le \underline{\sigma}^2, Z_{\mathcal{T}(Z)}=b\Big).
\end{align*}
Let us assume now that $x\in]a,\frac{a+b}{2}]$ and $h\le h_0$. By comparison, the trajectory of $X^x$ always stays below that of $X^{(a+b)/2}$.  Setting $r_h=(b-a)/(b_h-a)$ we get
\begin{align*}
\mathbb{P}\Big(  \tau^x_-&\le 1, X^x_{\tau^x_-}=a\Big)=\mathbb{P}\Big(  \mathcal{T}(r_hX^x+a(1-r_h))\le 1, X^x_{\tau^x_-}=a \Big)\\
&=\mathbb{P}\Big(  \mathcal{T}_a(r_hX^x+a(1-r_h))\le 1, X^x_{\tau^x_-}=a \Big)\\
&=\mathbb{P}\Big(  \mathcal{T}_a(X^x)\le 1, \mathcal{T}_a(X^x)<\mathcal{T}_{b}(r_hX^x+a(1-r_h)) \Big)\\
&\ge \mathbb{P}\Big(  \mathcal{T}_a(X^{(a+b)/2})\le 1, \mathcal{T}_a(X^{(a+b)/2})<\mathcal{T}_{b}(r_h X^{(a+b)/2}+a(1-r_h)) \Big)\\
&\ge \mathbb{P}\Big(  \mathcal{T}_a(X^{(a+b)/2})\le 1, \mathcal{T}_a(X^{(a+b)/2})<\mathcal{T}_{b}(r_{h_0} X^{(a+b)/2}+a(1-r_{h_0})) \Big)\\
&=:\kappa_1
\end{align*}
where $\kappa_1$ is a positive constant independent of both $h$ and $x$. Hence
\begin{equation}
\mathbb{P}(E_{ab})\ge \kappa_11_{]a,\frac{a+b}{2}]}(x)\Psi(h),\quad\mbox{with}\quad \Psi(h):=\mathbb{P}\Big(\mathcal{T}_b(Z)\le \underline{\sigma}^2, Z_{\mathcal{T}(Z)}=b\Big).
\label{psip}
\end{equation}
It suffices to lower bound the function $\Psi$ using scale functions and an independent exponential random variable which permits to relate the computation of $\Psi$ to a particular Laplace transform whose expression is explicit (see, \cite{borodin-salminen} p309).\\
Let $\mathcal{E}$ be an exponentially distributed random variable with parameter $\lambda$ and let $h_{\alpha} = he^{-\int^1_0\vert\alpha(u)\vert du}$. Then $\Psi(h)$ can be lower-bounded by the difference of $\Psi_1(h)$ and $\Psi_2(h)$:
\begin{align*}
\Psi(h)
&\geqslant \mathbb{P}_{a+h_{\alpha}}\Big(\mathcal{T}(Z)\le \mathcal{E},\mathcal{T}(Z) = \mathcal{T}_b(Z)\Big)\\
& \quad - \mathbb{P}_{a+h_{\alpha}}\Big(\mathcal{T}(Z) = \mathcal{T}_b(Z), \mathcal{T}(Z) \leqslant \underline{\sigma}^2, \mathcal{E} > \underline{\sigma}^2\Big)
&= \Psi_1(h) -\Psi_2(h).
\end{align*}
The first term of the r.h.s $\Psi_1(h)$ is evaluated as follows
\begin{align*}
\Psi_1(h)
&= \mathbb{E}_{a+h_{\alpha}}\Big[e^{- \lambda \mathcal{T}}\ 1_{\{\mathcal{T}(Z) = \mathcal{T}_b(Z)\}}\Big]= e^{\mu(b-a-h_{\alpha})} \frac{\sinh(h_{\alpha}\sqrt{2 \lambda + \mu^2})}{\sinh((b-a)\sqrt{{2 \lambda + \mu^2})}}\\
&\sim e^{\mu(b-a)} \frac{he^{-\int^1_0\vert\alpha(u)\vert du}\sqrt{2 \lambda + \mu^2}}{\sinh((b-a)\sqrt{{2 \lambda + \mu^2})}}, \mbox{ as } h \mbox{ tends to } 0.
\end{align*}
The second term $\Psi_2(h)$ has to be upper bounded:
\begin{align*}
\Psi_2(h)
&= \mathbb{P}_{a+h_{\alpha}}\Big(\mathcal{T}(Z) = \mathcal{T}_b(Z), \mathcal{T}(Z) \leqslant \underline{\sigma}^2,\Big)\mathbb{P}\Big(\mathcal{E} > \underline{\sigma}^2\Big)\\
&\leqslant \mathbb{P}_{a+h_{\alpha}}\Big(\mathcal{T}(Z) = \mathcal{T}_b(Z)\Big)e^{-\lambda \underline{\sigma}^2}.
\end{align*}
Using the scale function of the drifted Brownian motion, we obtain
\begin{align*}
\mathbb{P}_{a+h_{\alpha}}\Big(\mathcal{T}(Z) = \mathcal{T}_b(Z)\Big) &=  e^{-\int_0^1|\alpha(u)|du}\frac{e^{-2\mu (a + h_{\alpha})}- e^{-2\mu a}}{e^{-2\mu b} - e^{-2 \mu a}}\\
&\sim -h e^{-\int_0^1|\alpha(u)|du}\frac{2\mu e^{-2 \mu a}}{e^{-2\mu b} - e^{-2 \mu a}} \mbox{ as } h \mbox{ tends to } 0.
\end{align*}
If the parameter of the exponentially distributed r.v. becomes large then it is easy to prove that $\Psi_2(h)$ becomes negligible with respect to $\Psi_1(h)$. Consequently we can choose a particular value of $\lambda$ which leads to $2\Psi_2(h)\le \Psi_1(h)$ and therefore permits to bound $\Psi(h)$ by below.\\
Combining \eqref{psip} and the description of $\Psi(h)$, we manage to bound $\mathbb{P}(E_{ab})$ by below for $x\le (a+b)/2$. In the case where the starting point of the diffusion is in the lower part of the interval, we bound $\mathbb{P}(E_{ba})$ by below with the value $0$ which implies the existence of a strictly positive lower bound of $\mathbb{P}(E_{ab})+\mathbb{P}(E_{ba})$. In the other case (i.e. the starting point is in the upper part of the interval), we bound $\mathbb{P}(E_{ab})$ by below with the value $0$ and deal with $\mathbb{P}(E_{ba})$ in a similar way as previously described. In any case, the inequality \eqref{minF} leads to the existence of $\kappa>0$ such that
\begin{equation*}
\frac{\partial F}{\partial x}(t,x) \geqslant \kappa , \quad \forall (t,x)\in[2,\infty[\times[a,b].
\end{equation*}
\underline{Second case: $t\leqslant 2$.}\\
First we consider the derivative at the boundary of the interval $[a,b]$. Let us note that $F(t,a)=a$. Hence $\frac{\partial F}{\partial x}(t,a)=\lim_{h\to 0^+}\frac{1}{h}\,(\mathbb{E}_{a+h}[X_{t\wedge\tau_{ab}}]-a)$. Since we need a lower bound, we shall use a comparison result concerning the L-class diffusions. Proposition 3.14 leads to $Z_{\gamma(t)}^T\le X_t$ for all $t\le T$. We set here $T=2$ and $\mu_T$ is defined in the statement of the proposition. If $\mu_T\ge 0$ then we replace it by a strictly negative value and therefore the comparison result remains true. So we assume for the sequel that $\mu_T<0$. We deduce that 
\begin{equation}
\mathbb{E}_{a+h}[X_{t\wedge\tau_{ab}}]-a\ge \mathbb{E}_x[Z_{\gamma(t)\wedge\tau_{ab}}^T],
\label{minesp}
\end{equation}
 where $\tau_{ab}$ stands either for the exit time of $X$ either for the exit time of $Z$. Let us now consider the convex function $f(x)=e^{-2\mu_T x}$. It is well known that
\[
x-a\ge \frac{(b-a)}{f(b)-f(a)}\,(f(x)-f(a)), \forall x \in ]a,b[.
\]
As $f$ is the scale function of the drifted Brownian motion, $f(Z^T_t)$ is a martingale and the optimal stopping theorem leads to 
\begin{align*}
\mathbb{E}_x[Z^T_{\gamma(t) \wedge \tau_{ab}}-a]&\ge \frac{(b-a)}{f(b)-f(a)}\,\mathbb{E}_x\Big[e^{-2\mu_T Z^T_{\gamma(t) \wedge \tau_{ab}}} -e^{-2\mu_T a} \Big]\\
&=\frac{(b-a)}{f(b)-f(a)}\Big(e^{-2\mu_T x} -e^{-2\mu_T a} \Big).
\end{align*}
In particular, for $x=a+h$,
\begin{align*}
\mathbb{E}_{a+h}[Z^T_{t\wedge \tau_{ab}}-a]&\ge (b-a)e^{2\mu_T(b-a)}(e^{-2\mu_T h} -1)\\
&\sim -2\mu_T h(b-a)e^{2\mu_T(b-a)}, \mbox{ as h tends to } 0.
\end{align*}
We obtained the existence of a constant $\eta^a_T>0$ such that $\frac{\partial F}{\partial x}(t,a) \geqslant \eta^a_T$, for any $t \leqslant 2$. By similar arguments,  we can obtain $\frac{\partial F}{\partial x}(t,b) \geqslant \eta^b_T$, for all $t \leqslant 2$. Since $\frac{\partial F}{\partial x}(t,x)$ satisfies a second order parabolic PDE with regular coefficients, we can apply the maximum principle (see, for instance, \cite{Evans} or \cite{Friedman}). Consequently the minimum of the derivative on the domain $[0,2]\times[a,b]$ is reached at the boundary. Let us observe what happens on each side of this rectangle. For $x=a$ we have just proven that there exists a minimum which is strictly positive so is it for $x=b$. For $t=0$ the derivative is equal to $1$ and for $t=2$ the first part of the proof ensures the derivative to be minimized. To sum up, the derivative is lower bounded by a strictly positive constant on the whole rectangle  $[0,2]\times[a,b]$.
\end{proof}

\begin{proposition}
\label{distF}
There exists two constants $\kappa_a>0$ and  $\kappa_b > 0$ such that 
\begin{equation}
F(t,x) - a \leqslant \kappa_a (x-a) \text{   and   } b - F(t,x) \leqslant \kappa_b(b-x),
\end{equation}
for all $(t,x)\in\mathbb{R}_+\times[a,b]$.
\end{proposition}

\begin{proof}
Let us recall the probabilistic representation: $F(t,x)=\mathbb{E}[X^x_{t\wedge\tau_{ab}}]$.\\
We set $T=\gamma(1)$  and consider $(Z^T_t)$ the diffusion introduced in Remark \ref{inversecomparison} with initial condition $Z^T_0=X_0^x=x$. We construct a new continuous diffusion process $(Z_t)$ which is equal to $(Z_t^T)$ on the time interval $[0,1]$ and which satisfies the following SDE otherwise:
\begin{align*}
dZ_t=\displaystyle\Big( \frac{\alpha(\gamma(t))}{\tilde{\sigma}^2(\gamma(t))}\, Z_t+  \frac{\beta(\gamma(t))}{\tilde{\sigma}^2(\gamma(t))}\Big)\,dt+dW_t,\quad t>1.
\end{align*}
Extending the comparison result of Remark \ref{inversecomparison}, we know that  $Z_t\ge X_{\gamma(t)}$ for all $t\ge 0$. Hence
\[
F(t,x)-a\le\mathbb{E}_x[Z_{\gamma^{-1}(t)\wedge \tau_{ab}(Z)}- a].
\]
We split the study into two different cases :
\begin{itemize}
\item First case: $\gamma^{-1}(t)\le 1$. As in the proof of Proposition \ref{minderiv}, the function $f(x)=e^{-2\mu x}$ plays an important role since $f(Z_t)$ is a martingale for $t\le \gamma(1)$. Using twice the Lagrange mean theorem combined with the optional stopping theorem implies
\begin{align*}
F(t,x)-a &\le \eta_1 \mathbb{E}_x\Big[e^{-2\mu_T Z^T_{t\wedge \tau_{ab}}} -e^{-2\mu_T a} \Big]=\eta_1\Big(e^{-2\mu_T x} -e^{-2\mu_T a} \Big)\\
&\le \kappa_a(x-a),
\end{align*}
where $\kappa_a=\Big( \sup_{x\in[a,b]} f'(x) \Big)\Big( \inf_{x\in[a,b]} f'(x) \Big)^{-1}$.

\item Second case: $\gamma^{-1}(t)>1$. We decompose $F$ as follows
\begin{align*}
F(t,x)-a
&\le \mathbb{E}_x[(Z_{\gamma^{-1}(t)\wedge \tau_{ab}(Z)}-a)1_{\{ \tau_{ab}(Z)>1 \}}]\\
&+\mathbb{E}_x[(Z_{\gamma^{-1}(t)\wedge \tau_{ab}(Z)}-a)1_{\{ \tau_{ab}(Z)\le 1 \}}]\\
&\le (b-a)\mathbb{P}_x(\tau_{ab}(Z^T)>1)+\mathbb{E}_x[(Z^T_{1\wedge \tau_{ab}(Z^T)}-a)1_{\{ \tau_{ab}(Z^T)\le 1 \}}]\\
&\le (b-a)\mathbb{E}_x[\tau_{ab}(Z^\mu)]+\mathbb{E}_x[Z^\mu_{1\wedge \tau_{ab}(Z^\mu)}]-a.
\end{align*}
The expression $\mathbb{E}_x[Z^T_{1\wedge\tau_{ab}(Z^T)}]-a$ can be bounded using similar arguments (Lagrange's mean and optional stopping theorems) as those presented in the first part of the proof. Moreover, let us note that the function $g(x):=\mathbb{E}_x[\tau_{ab}(Z^T)]$ is solution (\cite{bass}, page 45, Theorem 1.2) of
\[
\frac{1}{2}\,g''+\mu_T g'=-1\quad
\mbox{for}\ x\in]a,b[\quad \mbox{and}\ g(a)=g(b)=0.
\]
We recall that $T=\gamma(1)$ and $\mu_T$ is defined in Remark \ref{inversecomparison}. 
The explicit solution of this equation is given by 

\begin{equation*}
g(x) = \frac{(b-a)(e^{-2 \mu_T a} - e^{-2 \mu_T x})}{\mu_T(e^{-2 \mu_T a} - e^{-2 \mu_T b})} - \frac{(x-a)}{\mu_T}.
\end{equation*} 
Applying once again Lagrange's mean theorem, we obtain the existence of a constant $C_g>0$ such that  $g(x) \leqslant C_g(x-a)$ for all  $x\in [a,b]$. Using similar arguments (just replacing Remark \ref{inversecomparison} by Proposition \ref{comparison}), we also prove that $b - F(t,x) \leqslant \kappa_b(b-x)$.
\end{itemize}\end{proof}

\subsubsection*{Proof of Theorem \ref{meanL}.}

We already presented all the necessary ingredients in order to prove the statement of Theorem  \ref{meanL} which concerns the average number of steps.

\begin{proof}
Our choice for the bound of the average number of steps is based on the martingale theory.  We recall that $F$ is defined by \eqref{defF} and introduce another important function $H$ defined by $H=V\circ F$ with
\begin{equation}
\label{defV}
V(x)=\log\left(  \frac{(x-a)(b-x)}{\gamma\epsilon(b-a-\gamma\epsilon)} \right).
\end{equation}
Let us note that $V$ is non negative on the whole interval $[a+\gamma\epsilon,b-\gamma\epsilon]$. Since $F$ is a the solution of \eqref{edp}, the function $H$ just introduced satisfies the following partial differential equation:
\begin{equation}
\frac{\partial H}{\partial t } + (\alpha(t)x + \beta(t))\frac{\partial H}{\partial x} + \frac{1}{2}\tilde{\sigma}(t)^2\frac{\partial^2 H}{\partial x^2} = \frac{1}{2}\tilde{\sigma}(t)^2V''(F(t,x))\left(\frac{\partial F}{\partial x}(t,x)\right)^2.
\label{edpbis}
\end{equation}
Let us also recall that $(T_n,X_n)$ defined in \eqref{defTn} is the sequence of successive exit times and exit positions issued from {\sc Algorithm}$_m$.\\ 
We focus our attention on the sequence  $Z_n=H(X_n)+G(n)$ with  $G(0)=0$. Here $G$ stands for a positive function, we are going to precise this function in the sequel. This stochastic process is a super-martingale with respect to the Brownian filtration $(\mathcal{F}_{T_n})_{n \in \mathbb{N}}$.
Using Itô's formula and the partial differential equation satisfied by $H$, we obtain  for $\mathcal{D}_{n}:=\mathbb{E}[Z_{n+1}-Z_n|\mathcal{F}_{T_n}]$, 
\begin{align*}
\mathcal{D}_n &=  \mathbb{E}\left[\int^{T_{n+1}}_{T_n} \frac{\partial H}{\partial t }(s,X_s) + (\alpha(s)X_s + \beta(s))\frac{\partial H}{\partial x}(s,X_s)\right. \\
&\quad + \left.\left.\frac{1}{2}\tilde{\sigma}(s)^2\frac{\partial^2 H}{\partial x^2}(s,X_s)ds \right\vert \mathcal{F}_{T_n}\right]\\
&\quad +\mathbb{E}[M_{n+1} - M_n \vert \mathcal{F}_{T_n}] +(G(n+1)- G(n))\\
&= \mathbb{E}\left.\left[\int^{T_{n+1}}_{T_n} \frac{1}{2}\tilde{\sigma}(s)^2 V''(F(s,X_s))\left(\frac{\partial F}{\partial x}(s,X_s)\right)^2ds \right\vert \mathcal{F}_{T_n}\right]\\
& \quad + (G(n+1)- G(n)),
\end{align*}
where $(M_n)_{n \in \mathbb{N}}= \left( \int^{T_n}_0 \tilde{\sigma}(s)\frac{\partial H}{\partial x}(s,X_s) dW_s\right)_{n \in \mathbb{N}}$ is a martingale.
Using Proposition \ref{minderiv}, Proposition \ref{distF} and the lower bound $\underline{\sigma}$ of $\tilde{\sigma}$ we obtain

\begin{equation}
\mathcal{D}_n  \le -\frac{1}{2}\underline{\sigma}^2\kappa^2(\mathcal{I}(a)+\mathcal{I}(b))+G(n+1)-G(n),  
\label{martI}
\end{equation}
where $\mathcal{I}(x)=\mathbb{E}\left[\left. \int_{T_n}^{T_{n+1}}\frac{1}{\kappa_x^2(X_s-x)^2}\, ds   \right| \mathcal{F}_{T_n}\right]$.\\
We aim to bound by below the previous integral by considering the shape of the $n^{th}$ spheroid:
\begin{align}
\psi_+^L(t)-a_{\gamma,X_{n}} &\leqslant  d_n \Delta_m + X_{n} -a_{\gamma,X_{n}}\nonumber\\
&\leqslant \min(1, \kappa_-)(X_{n}-a_{\gamma,X_{n}}) + X_{n} -a_{\gamma,X_{n}}\nonumber\\
&\le 2(X_{n}-a).
\label{majboundLa}
\end{align}
%
%
%
\noindent This bound implies
\begin{align*}
\mathcal{I}(a) &\geqslant \mathbb{E}\left.\left[\int^{T_{n+1}}_{T_n} \frac{ds}{4\kappa_a^2(X_{n}-a)^2}\right\vert \mathcal{F}_{T_n}\right]= \mathbb{E}\left.\left[ \frac{T_{n+1}-T_n}{4\kappa_a^2(X_{n}-a)^2}\right\vert \mathcal{F}_{T_n}\right]\\
&= \mathbb{E}\left.\left[ \frac{\rho_L^{-1}(\rho_L(T_n)+\tau_{n+1})}{4\kappa_a^2(X_{n}-a)^2}\right\vert \mathcal{F}_{T_n}\right]
\end{align*}
where $\tau_{n+1}$ is the Brownian exit time from the spheroid of parameter size $d_n$.
\begin{equation*}
\mathcal{I}(a) \ge \mathbb{E}\left.\left[ \frac{\tau_{n+1}}{4\kappa_a^2 r_n (X_{n}-a)^2}\right\vert \mathcal{F}_{T_n}\right]
\end{equation*}
where $r_n$ is the maximum of the derivative $\rho'$ on the time interval $[T_n,T_n+m]$ which contains $[T_n, \rho^{-1}_L(\rho_L(T_n)+\tau_{n+1}) ]$.
We note that $\tau_{n+1}\sim d^2_n\tau$ where $\tau$ denotes the Brownian exit time from the Brownian spheroid of parameter $1$. Hence
\begin{equation*}
\mathcal{I}(a)
\ge\frac{d_n^2}{4\kappa_a^2 r_n (X_{n}-a)^2}\ \mathbb{E}[\tau].
\end{equation*}
Similarly to \eqref{majboundLa} we have $b_{\gamma,X_n}-\psi^L_-(t)\le 2(b-X_n)$ and the same arguments just presented lead to
\begin{equation*}
\mathcal{I}(b)=\mathbb{E}\left.\left[\int^{T_{n+1}}_{T_n} \frac{ds}{\kappa_b^2(b -X_s)^2} \right\vert \mathcal{F}_{T_n}\right] \geqslant \frac{d_n^2}{4\kappa_b^2 r_n (b - X_{n})^2}\ \mathbb{E}[\tau].
\end{equation*}
Setting $\kappa_{ab} = \max(\kappa_a,\kappa_b)$, we obtain

\begin{equation*}
\mathcal{D}_n \leqslant -\frac{d_n^2}{r_n \kappa_{ab}^2}\ \mathbb{E}[\tau]\left(\frac{1}{(b - X_{n})^2} +\frac{1}{( X_{n}-a)^2}\right)+ G(n+1)- G(n).
\end{equation*}
Let us first consider the case: $X_{n}-a \leqslant b - X_{n} $ (the other case can be studied in a similar way, it suffices to replace $X_n-a_{\gamma,X_{n}}$ by $b_{\gamma,X_{n}}- X_n$). Then $d_n= \frac{\min(1, \kappa_-)}{\Delta_m}(X_{n}-a_{\gamma,X_{n}})$ and
\begin{align*}
\mathcal{D}_n &\leqslant -2\frac{d_n^2}{r_n \kappa_{ab}^2}\ \mathbb{E}[\tau]\frac{1}{(X_{n}-a)^2} + G(n+1)- G(n)\\
&\leqslant -\frac{\min(1, \kappa_-)^2}{r_n \Delta_m^2\kappa_{ab}^2}\ \mathbb{E}[\tau] + G(n+1)- G(n).
\end{align*}
We finally find $G$ by seeking a lower bound of $\frac{\min(1, \kappa_-)^2}{r_n\Delta_m^2}$.
We consider two different cases:\\
First case: $\kappa_-\ge1$. We introduce $\alpha_n$, $\beta_n$ and $\tilde{\sigma}_n$ the maximum of $|\alpha |$ respectively $|\beta|$ and $\tilde{\sigma}$ on the time interval $[0,nm]$. The definition of $\Delta_m$ given by \eqref{deltam} and the definition of $\rho$ in Lemma \ref{frontL} lead to
\begin{align*}
\Delta_m^2 r_n &\leqslant e^{4\int_{T_n}^{T_n+m}|\alpha(s)|\,ds}\tilde{\sigma}_n^2 \left( \frac{1}{\sqrt{e}} +\sqrt{\int^{T_n+m}_{T_n} \frac{\vert \beta(s)\vert^2}{\tilde{\sigma}(s)^2}ds}\right)^2\\ 
 &\leqslant e^{4m \alpha_n}\tilde{\sigma}_n^2 \left( \frac{1}{\sqrt{e}} +\sqrt{m} \frac{\beta_n}{\underline{\sigma}}\right)^2. 
\end{align*}
For the other case: $\kappa_-<1$
\begin{align*}
\frac{\min(1, \kappa_-)^2}{r_n \Delta_m^2}&\geqslant \frac{\rho(T_n + m) - \rho(T_n)}{r_n (b-a)^2}= \frac{\int_0^m \rho'(T_n +s)ds}{r_n(b-a)^2}.
\end{align*}
Using the definitions of $\rho$, $r_n$ and the continuity of $\tilde{\sigma}$, there exists $t_0\in[T_n,T_n+m]$ such that $r_n=\rho'(t_0)$ and therefore
\begin{align*}
\frac{\rho'(T_n+s)}{r_n}&=\frac{\tilde{\sigma}^2(T_n+s)}{\tilde{\sigma}^2(t_0)}\, e^{-2\int_{t_0}^{T_n+s}\alpha(u)\,du}\ge \frac{\tilde{\sigma}^2(T_n+s)}{\tilde{\sigma}^2(t_0)}\, e^{-2|T_n+s-t_0|\alpha_n}\\
&\ge \frac{\tilde{\sigma}^2(T_n+s)}{\tilde{\sigma}^2(t_0)}\, e^{-2m\alpha_n}.
\end{align*}
Since $\tilde{\sigma}$ satisfies Assumption \ref{assump2}, we obtain the following lower bound by integrating with respect to the variable $s$, 
\begin{align*}
\frac{\min(1,\kappa_-)^2}{r_n\Delta_m^2}&\ge \frac{m \chi_m}{(b-a)^2}e^{-2m\alpha_n}. 
\end{align*}
Denoting $\zeta_{n+1}$ the minimum of the two quantities previously computed, we define recursively the sequence $G(n)$ by
\begin{equation*}
G(n+1) - G(n) = \zeta_{n+1},\quad \forall n \geqslant 0 ,\quad \mbox{and}\ G(0)=0.
\end{equation*}
The sum of these increments leads to

\begin{equation*}
\sum_{i=0}^{n-1} G(i+1) - G(i) =\sum_{i=1}^{n} \zeta_{i} = G(n) - G(0) = G(n).
\end{equation*}
For any parameter $\tilde{q}>q$, Assumption \ref{assump1} implies the existence of a constant $\tilde{C}>0$ independent of $\epsilon$ such that
\begin{equation}
G(n) \geqslant \frac{1}{\tilde{C}}\sum_{k=1}^n \frac{1}{k^{\tilde{q}}}\geqslant\frac{1}{\tilde{C}(1-\tilde{q})}(n^{1-\tilde{q}} - 1),\quad \forall n\ge 1.
\label{minG}
\end{equation}
Moreover the particular choice of the function $G$ permits to obtain $\mathcal{D}_n\le 0$ for all $n$. Consequently $Z_n=H(n,X_n)+G(n)$ is a super-martingale. A generalization of Proposition \ref{potential} permits to obtain the upper bound
\begin{equation}
\label{minmeanG}
\mathbb{E}[G(N_\epsilon)]\le H(0,x_0)=V\circ F(0,x_0)=V(x_0).
\end{equation}
Combining \eqref{minG}, \eqref{minmeanG} and the definition of the function $V$ in \eqref{defV} leads to
\begin{equation*}
\mathbb{E}[N_{\epsilon}^{1-\tilde{q}}] \leqslant \tilde{C}(1-\tilde{q})\log\left( \frac{(x_0-a)(b-x_0)}{\gamma\epsilon(b-a-\gamma\epsilon)} \right)+1.
\end{equation*}
This bound corresponds to the announced result. In order to conclude the proof, we just need to precise that $N_\epsilon$ is a.s. finite, see Lemma \ref{Nfinite}. Such a condition is required to apply the generalization of Proposition \ref{potential}.
\end{proof}

\begin{lemma}
\label{Nfinite}
 The stopping procedure $N_\epsilon$ of {\sc Algorithm}$_{m}$ is a.s. finite. Moreover the outcome of the algorithm $\mathcal{T}_\epsilon$ is stochastically upper bounded by $\mathcal{T}$, the diffusion first exit time.
\end{lemma}
\begin{proof}
\textit{Step 1.} We emphasize a link between a sample of a L-class diffusion process and the Markov chain generated by the algorithm, denoted $((T_n, X_n))_{n\in \mathbb{N}}$ with $(T_0, X_0) = (0,0)$.\\
Let us consider a sample of a L-class diffusion. At the starting point of this path, we create a spheroid of maximal size which belongs to the set $[a,b] \times \mathbb{R}_+$. The first intersection point of this spheroid and the path gives us a first point $(t_1,z_1)$. This construction implies that $(t_1,z_1)$ and $(T_1,X_1)$ are identically distributed. Then considering $(t_1,z_1)$ as a new starting point we construct a spheroid of maximal size and denote by $(t_2,z_2)$ the first intersection point between this new spheroid and the diffusion path starting in $(t_1,z_1)$. Once again we get by construction that  $(t_2,z_2)$ and $(T_2,X_2)$ are identically distributed. We build step by step a sequence $((t_n, z_n))_{n\in \mathbb{N}}$ of intersections between the considered sample and the spheroids in such a way that the sequences $((t_n,z_n))_{n\ge 0}$ and $((T_n,X_n))_{n\ge 0}$ are identically distributed.\\
If we introduce $N_{\epsilon}$ the stopping time appearing in the stopping procedure of the algorithm  and $\tilde{N}_{\epsilon} = \inf\{n \in \mathbb{N}, z_n \notin [a + \epsilon, b - \epsilon]\}$, the identity in law of those random variables holds. By construction, $t_n \leqslant T$ for all $n \in \mathbb{N}$, where $T$ stands for the diffusion first exit time from the interval $[a,b]$. This inequality remains true when $t_n$ is replaced by the random stopping time $t_{\tilde{N}_\epsilon}$.\\
Since $t_{\tilde{N}_\epsilon}$ and $t_{N_\epsilon}$ are identically distributed, we deduce that the outcome  of {\sc Algorithm}$_m$ is stochastically smaller than $T$.\\ 
\textit{Step 2.}
We prove now that $N_\epsilon$ is a.s. finite. Using \eqref{defTn} and \eqref{reltempsL} we obtain

\begin{equation*}
T_n = \rho^{-1}_L(d_1^2 \tau_1 + d_2^2 \tau_2 +\ldots+ d_n^2 \tau_n),   
\end{equation*}
where $(\tau_k)_{k\geqslant 1}$ is a sequence of independent Brownian exit times from the unit spheroid and $d_k$ represents the size of the spheroid \eqref{paraml} starting in $(T_k, X_k)$ and included in $[a,b]$. Let $t_0 > 0$. Then
\begin{align*}
\mathbb{P}(T_n \leqslant t_0) &= \mathbb{P}(d_1^2 \tau_1 + d_2^2 \tau_2 +\ldots+ d_n^2 \tau_n \leqslant \rho^{-1}_L(t_0))\\
&\leqslant \mathbb{P}\left(\tau_1 + \tau_2 +\ldots + \tau_n \leqslant \frac{\rho^{-1}_L(t_0)}{\underline{d}(t_0)} \right),
\end{align*}
where $\underline{d}(t_0)$ is defined by 
\begin{equation*}
\underline{d}(t_0) = \inf\limits_{x \in [a+\epsilon, b - \epsilon],\ t \leqslant t_0} d(x,t) >0.
\end{equation*}
Since $\displaystyle{\sum_{k+1}^n \tau_k}$ tends to $+\infty$ a.s.,
\begin{equation*}
\lim\limits_{n \rightarrow + \infty} \mathbb{P}(T_n \leqslant t_0) = \mathbb{P}(T_{\infty} \leqslant t_0) = 0, \quad \forall t_0 > 0.
\end{equation*}
We deduce that $\lim_{n \rightarrow +\infty} T_n = + \infty$ a.s. Combining this limiting result to the first step of the proof, that is $T_n \overset{(d)}{\leqslant} T$, implies: $N_\epsilon < +\infty$ a.s.
\end{proof}

\subsubsection{Rate of convergence}

The second important result in the study of the algorithm is the description of the  rate of convergence. It is of prime interest to known how close the outcome of the algorithm and the exit time of the L-class diffusion are. The convergence result is essentially based on the strong relation between the Brownian motion and the L-class diffusion.

\begin{thm}
\label{errorL}
Let us denote by $\overline{\alpha}_t$ (respectively $\overline{\beta}_t$) the maximal value of the function $|\alpha|$ (resp. $|\beta|$) on the interval $[0,t]$. We also introduce $F$ the cumulative distribution
function of the L-class diffusion exit time from the interval $[a, b]$ and $F_\epsilon$ the distribution function of the algorithm
outcome. Then, for any $t\ge 0$ and any $\rho>1$ there exists $\epsilon_0>0$ such that 

\begin{equation}
\left( 1-\rho\sqrt{\epsilon}\ \frac{1+\overline{\beta}_t}{\underline{\sigma}} \right)F_{\epsilon}(t -\epsilon) \leqslant F(t) \leqslant F_{\epsilon}(t),\quad\forall \epsilon\le \epsilon_0,
\end{equation}
the constant $\underline{\sigma}$ being defined in \eqref{minsigma}. Moreover this convergence is uniform on each compact subset of the time axis.
\end{thm}

\begin{proof}
As in Lemma \ref{Nfinite}, we build step by step a sequence $((t_n, z_n))_{n\in \mathbb{N}}$ of intersections between the path of the L-class diffusion process and the spheroids in such a way that the sequences $((t_n,z_n))_{n\ge 0}$ and $((T_n,X_n))_{n\ge 0}$ are identically distributed.\\
If we introduce $N_{\epsilon}$ the stopping time appearing in the stopping procedure of the algorithm  and $\tilde{N}_{\epsilon} = \inf\{n \in \mathbb{N}, z_n \notin [a + \epsilon, b - \epsilon]\}$, the identity in law of those random variables holds. By construction, $t_n \leqslant \mathcal{T}$ for all $n \in \mathbb{N}$, where $\mathcal{T}$ stands for the diffusion first exit time from the interval $[a,b]$. This inequality remains true when $t_n$ is replaced by the random stopping time $t_{\tilde{N}_\epsilon}$. Hence
\begin{align}
1 - F(t) &= \mathbb{P}(\mathcal{T}>t) =\mathbb{P}(\mathcal{T}>t, t_{\tilde{N}_{\epsilon}}\leqslant t -\delta) + \mathbb{P}(\mathcal{T}>t, t_{\tilde{N}_{\epsilon}}> t -\delta)\nonumber\\
&\leqslant \mathbb{P}(\mathcal{T}>t, t_{\tilde{N}_{\epsilon}}\leqslant t -\delta) +1 - F_{\epsilon}(t- \delta),\quad \forall t\ge 0.
\label{majLerror}
\end{align}
We focus our attention on the first term of the r.h.s. Using the strong Markov property, we obtain
\begin{equation}
\mathbb{P}(\mathcal{T}>t, t_{\tilde{N}_{\epsilon}}\leqslant t -\delta) \leqslant F_{\epsilon}(t-\delta) \sup\limits_{(y,s) \in ([a,a+ \epsilon] \cup [b-\epsilon,b]) \times [0, t-\delta]} \mathbb{P}_{(y,\tau)}(\mathcal{T}>\delta).
\label{markovbis}
\end{equation}
Let us consider the case $y\in[b-\epsilon,b]$ (the study of the other case $y\in[a,a+\epsilon]$ is left to the reader since it suffices by symmetry to use exactly the same arguments). We first note that, for any $y\in[b-\epsilon,b]$,
\[
\mathbb{P}_{(y,s)}(\mathcal{T}>\delta)\le \mathbb{P}_{(y,s)}(\mathcal{T}_b>\delta)\le \mathbb{P}_{(b-\epsilon,s)}(\mathcal{T}_b>\delta),
\]  
where $\mathcal{T}_b$ stands for the first passage time through the level $b$. Let us introduce several notations: we denote the translated function $\alpha_s(t):=\alpha(s+t)$ (similar definitions for $\tilde{\sigma}_s$, $\beta_s$ and $\rho_s$ are defined by using the translated functions in \eqref{defrho}). 
The diffusion process on the time interval $[s,s+\delta]$ can be expressed using these translated functions. The condition $\mathcal{T}_b>\delta$ is equivalent to $\sup_{0\le r\le \delta} X_{s+r}<b$ and becomes, for all $r\le \delta$,
\begin{equation}
b-\epsilon+e^{2\int_0^r\alpha_s(u)\,du}W_{\rho_s(r)}+e^{\int_0^r\alpha_s(u)\,du}\int_0^r \beta_s(u)e^{-\int_0^u\alpha_s(w)\,dw}\,du < b.
\label{berror}
\end{equation}
Since $s\in[0,t-\delta]$ and $r\le \delta$, we obtain the following bound:
\[
\rho_s(\delta)\ge \underline{\sigma}^2\,\frac{1-e^{-2\overline{\alpha}_t \delta }}{2\overline{\alpha}_t}.
\]
The inequality \eqref{berror} implies
\[
\frac{1}{\sqrt{\rho_s(\delta)}} \sup_{0\le r\le \delta}W_{\rho_s(r)}\le \frac{e^{2\overline{\alpha}_t \delta}}{\sqrt{1-e^{-2\overline{\alpha}_t \delta}}} \frac{\sqrt{2\overline{\alpha}_t}}{\underline{\sigma}}(\epsilon+\overline{\beta}_t \delta)\le e^{3\overline{\alpha}_t \delta}\,\frac{\epsilon+\overline{\beta}_t\delta}{\underline{\sigma}\sqrt{\delta}}.
\]
The D\'esir\'e-Andr\'e reflexion principle for the Brownian motion implies that the l.h.s of the previous inequality has the same distribution than the absolute value of a standard gaussian random variable: $|G|$. Hence, for any $y\in[b-\epsilon,b]$ and for any $s\le t-\delta$:
\begin{equation}\label{majP}
\mathbb{P}(\mathcal{T}_b>\delta)\le \mathbb{P}\left(|G|\le  e^{3\overline{\alpha}_t \delta}\,\frac{\epsilon+\overline{\beta}_t\delta}{\underline{\sigma}\sqrt{\delta}} \right)\le \sqrt{\frac{2}{\pi}}\ e^{3\overline{\alpha}_t \delta}\,\frac{\epsilon+\overline{\beta}_t\delta}{\underline{\sigma}\sqrt{\delta}}.
\end{equation}
It suffices to choose $\delta=\epsilon$ in the previous inequality and to combine with \eqref{majLerror} in order to prove the statement of the theorem.
\end{proof}

\section{WOMS algorithm for G-class diffusions}

In this section we present an application of the results obtained so far to another family of diffusion processes: the growth processes (G-class). We shall just point out the existence of a strong link between linear and growth diffusions. 
\begin{definition}(G-class diffusions)
We call G-class diffusion any solution of
\begin{equation}
dX_t = (\alpha(t)X_t + \beta(t)X_t \log(X_t))dt + \tilde{\sigma}(t)dW_t,\quad X_0=x_0,
\label{edsg}
\end{equation}
where $\alpha$ and $\beta$ are real continuous functions and $\tilde{\sigma}$ is a continuous non-negative function.
\end{definition}

We first notice that this kind of process is non negative due to the logarithm function.
As for the $L$-class diffusions case, it is possible to emphasize an explicit expression of the solution of \eqref{edsg}. Here, the desired form is:
\begin{equation}
X_t = x_0\ G(t,W_{\gamma(t)}), \quad \forall t \geqslant 0.
\label{Gexplicit}
\end{equation}
The function $G$ is described in the following statement.
\begin{proposition}
The solution of the SDE \eqref{edsg} is given by \eqref{Gexplicit} with 
\begin{align*}
G(t,x) &= C(t) e^{\frac{\tilde{\sigma}(t)}{\sqrt{\gamma'(t)}}x}\\ 
&\text{ with } C(t)= \exp\left(e^{\int^t_0 \beta(s)ds}\int^t_0 \Big(\alpha(s) - \frac{1}{2}\tilde{\sigma}(s)^2\Big)e^{-\int_0^s \beta(u)du}ds\right)\\
&\text{ and }  \gamma(t) = \int^t_0 \tilde{\sigma}(s)^2e^{-2\int_0^s \beta(u)du} ds.
\end{align*}

\end{proposition}

This statement is an immediate consequence of the link built between the linear and the growth diffusions:

\begin{proposition}
If $X$ is solution of

\[
\left \{
\begin{array}{c @{=} l}
    dX_t &  (\alpha(t)X_t + \beta(t))dt + \sigma(t)dW_t\\
     X_0 & x_0 \\
\end{array}
\right.
\]
then $Y_t = e^{X_t}$ is solution of 

\begin{equation}
\left \{
\begin{array}{c @{=} l}
    dY_t &  (\tilde{\alpha}(t)Y_t + \tilde{\beta}(t)Y_t\log(Y_t))dt + \tilde{\sigma}(t)Y_tdW_t\\
     Y_0 & y_0 \\
\end{array}
\right.
\label{linkLG}
\end{equation}
with $\tilde{\alpha}(t) = \beta(t) + \frac{1}{2}\sigma(t)^2$, $\tilde{\beta}(t) = \alpha(t)$, $\tilde{\sigma}(t) = \sigma(t)$ and $y_0 = e^{x_0}$.
\end{proposition}

Hence, we manage to create a link between a solution of a $L$-class diffusion equation with $\alpha$, $\beta$, $\sigma$ and a solution of a $G$-class diffusion equation with $\tilde{\alpha}$, $\tilde{\beta}$, $\tilde{\sigma}$. 
\begin{proof}
To prove this statement, we apply Itô's formula

\begin{equation*}
Y_t = e^{X_t} = e^{X_0} + \int^t_0 e^{X_s} dX_s + \frac{1}{2}\int^t_0 e^{X_s} d\langle X,X \rangle_s
\end{equation*}
Hence, using the particular form of $X_t$ we obtain
\begin{align*}
Y_t &= Y_0 + \int^t_0 Y_s (\alpha(s)X_s + \beta(s))ds + \int^t_0\sigma(s)dB_s + \frac{1}{2}\int^t_0 Y_s \sigma(s)^2 ds\\
&= Y_0 + \int^t_0 (Y_s (\beta(s) + \frac{1}{2}\sigma(s)^2) + Y_s \log(Y_s)\alpha(s))ds + \int^t_0 Y_s \sigma(s) dB_s\\
&= Y_0 + \int^t_0 (Y_s \tilde{\alpha}(s) + Y_s \log(Y_s)\tilde{\beta}(s))dt + \int^t_0 Y_s \tilde{\sigma}(s)dB_s.
\end{align*}
%

\end{proof}

We consider the exit time from the interval $[a,b]$, $a,b \in \mathbb{R}^+_*$ for a $G$ class-diffusion. The previous link established permits to focus our attention on the exit time from the interval $[\log(a),\log(b)]$ for $L$-class diffusion processes with modified coefficients.\\

We present now an adaptation of the WOMS algorithm which permits to approximate the exit time for G-class diffusions. In such a context we aim to describe the procedure, the averaged number of steps and the convergence rate. \\
\emph{The procedure.} Let us consider $(X_t)_{t\ge 0}$ the unique solution of the stochastic differential equation \eqref{edsg}. In order to approximate the first diffusion exit time $\mathcal{T}$ of the interval $[a,b]$ we introduce the linear diffusion $(Y_t)$ solution of \eqref{linkLG}. 
Since the exit time of the growth process $(X_t)$ from the interval $[a,b]$ and the exit time of the linear diffusion $(Y_t)$ from the interval $[\log(a),\log(b)]$ are identically distributed, we use {\sc Algorithm}$_m$ with a parameter $\epsilon$ small enough, with boundaries $\log(a)$ and $\log(b)$. 
As a immediate consequence, Theorem \ref{meanL}  points out the logarithmic upper-bound of the average number of steps and Theorem \ref{errorL} emphasizes the convergence rate of the algorithm outcome.

\section{Numerical application}

In order to illustrate the efficiency of {\sc Algorithm}$_{m}$, we present numerical results associated to a particular linear diffusion.
Let us consider $(X_t)_{t \ge 0}$ the solution of \eqref{edsl} with 
\begin{equation*}
\label{sinL}
\alpha(t) = \frac{\cos(t)}{2 + \sin(t)}, \quad \beta(t) = \cos(t), \quad \tilde{\sigma}(t) = 2 + \sin(t).
\end{equation*}
Let us just notice that $\alpha$ satisfies $\alpha(t)= \frac{\tilde{\sigma}'(t)}{\tilde{\sigma}(t)}$,
 such a property simplifies the link between the diffusion process and a standard one-dimensional Brownian motion.
%
%
%
In particular, we obtain a simple expression of the time change appearing in \eqref{eqf}: $\rho(t)=4t$. Indeed \eqref{defrho} implies
\begin{align*}
\rho(t) 
&=\int^t_0 (2 + \sin(s))^2e^{-2\int^s_0 \frac{\cos(u)}{2 + \sin(u)}du}ds\\
&=\int^t_0 (2 + \sin(s))^2e^{-2(\log(2 + \sin(s))-\log(2))}ds= 4 t.
\end{align*}
Using Proposition \ref{spheroidL}, we determine the frontiers of the typical spheroid used in {\sc Algorithm}$_{m}$.

\begin{proposition}
If we denote by $\psi_{\pm}^L(t;t_0,X_{t_0})$ the spheroid starting in $(t_0,X_{t_0})$, we obtain
\begin{align}
\psi_{\pm}^L(t;t_0,X_{t_0}):= &\frac{2 + \sin(t+t_0)}{2 }\Big(\psi_{\pm}(4 t) + 2\log\Big(\frac{2 + \sin(t+t_0)}{2 + \sin(t_0)}\Big) \Big)\nonumber \\
&\quad + \Big(\frac{2 + \sin(t+t_0)}{2 + \sin(t_0)}\Big) X_{t_0}.
\label{frontsin}
\end{align}
and the exit time 
 $\tau^{t_0} = \inf\{t >0:\  X_t \notin [\psi^L_-(t;\,t_0,X_{t_0}),\psi^L_+(t;\,t_0,X_{t_0})]\}$ satisfies
\begin{equation}
\tau^{t_0} \overset{d}{=} \frac{1}{4}\tau
\end{equation}
where $\tau =\inf\{ t > 0:\ W_t \notin [\psi_-(t),\psi_+(t)]\}$.
\end{proposition}
The random walk on spheroids is therefore built using the typical boundaries \eqref{frontsin}. At each step of the algorithm, we need to use a scale parameter $d$ in order to shrink or enlarge the spheroid size in such a way that the domains always stay in the interval $[a,b]$. The general statement concerning the scale parameter \eqref{paraml} can be improved for this particular example.

Let $m > 0$ and $0<\gamma<1$. We recall that  $a_{\gamma,x_0}$ and $b_{\gamma,x_0}$ are defined by $a_{\gamma,x} = a+ \gamma(x-a)$ and $b_{\gamma,x} = b- \gamma(b-x)$. We choose the scale parameter $d$ in such a way that it satisfies
\[
d=\left \{
\begin{array}{c @{\text{ if }} l}
    \frac{\min(1, \kappa_+)}{\Delta_m}(b_{\gamma,x_0}-x_0) &  b-x_0 \leqslant x_0-a\\[5pt]
    \frac{\min(1, \kappa_-)}{\Delta_m}(x_0-a_{\gamma,x_0})& x_0 -a \leqslant b-x_0
\end{array}
\right.
\]
with
\begin{equation*}
\Delta_m = \frac{3}{2}\left(\frac{1}{\sqrt{e}}+(1+ \max(\vert a \vert, \vert b \vert))\sqrt{m}\right)
\end{equation*}
and $\kappa_\pm$ are defined by the following equations:
\begin{equation*}
 \kappa_+(b_{\gamma,x_0}-x_0) = 2\Delta_m \sqrt{m} \text{ and } \kappa_-(x_0 -a_{\gamma,x_0}) = 2 \Delta_m \sqrt{m}.
\end{equation*}
We just note that this particular value $\Delta_m$ is an easy upper-bound of the parameter emphasized in \eqref{deltam}. We just adapted the choice of the parameters to the particular diffusion studied in this section. Even if the procedure is close to the method presented in Proposition \ref{parameterL}, we notice that such a particular choice of $\Delta_m$ permits to point out a specific value $m$ such that both $\min(1, \kappa_-)$ and $\min(1, \kappa_+)$ are equal to $1$. This value corresponds to
\begin{equation*}
 m= \left(\frac{\sqrt{\frac{1}{e}+\frac{4}{3}(b-a)(1+ \max(\vert a \vert, \vert b \vert)}-\frac{1}{\sqrt{e}}}{2(1+ \max(\vert a \vert, \vert b \vert))}\right)^2.
\end{equation*}

Using {\sc Algorithm}$_{m}$ as in Section 2.3 permits to approximate the first diffusion exit time from the interval $[a,b]$, see Figure \ref{algosin} and Figure \ref{sinresult}.
\begin{figure}[H]
   \centerline{\includegraphics[scale=0.25]{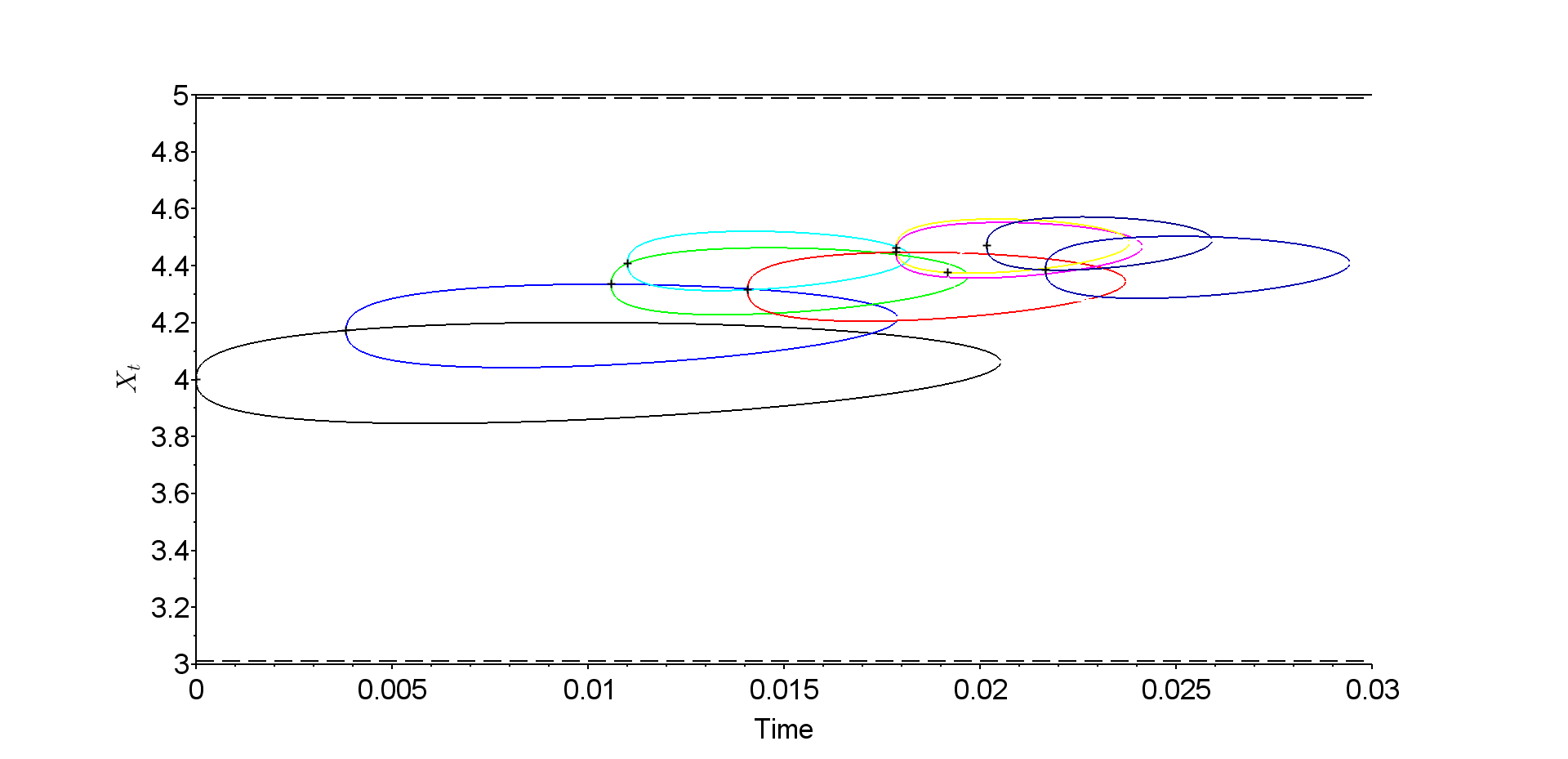}}
   \caption{\small A sample of  {\sc Algorithm}$_m$ for the diffusion process starting at $x=4$ in the interval $[3,5]$ with $\epsilon = 10^{-2}$ and $\gamma = 10^{-4}$ .}
   \label{algosin}
\end{figure}

\begin{figure}[H]
  \centerline{\includegraphics[width=6.5cm]{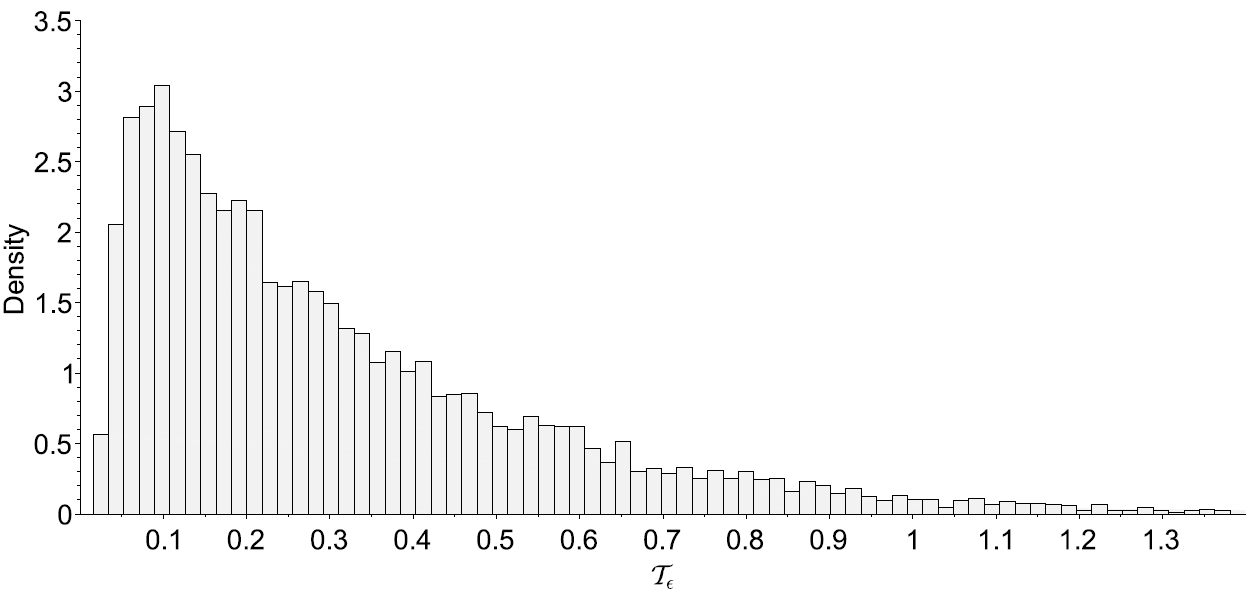}  \includegraphics[width=7.0cm]{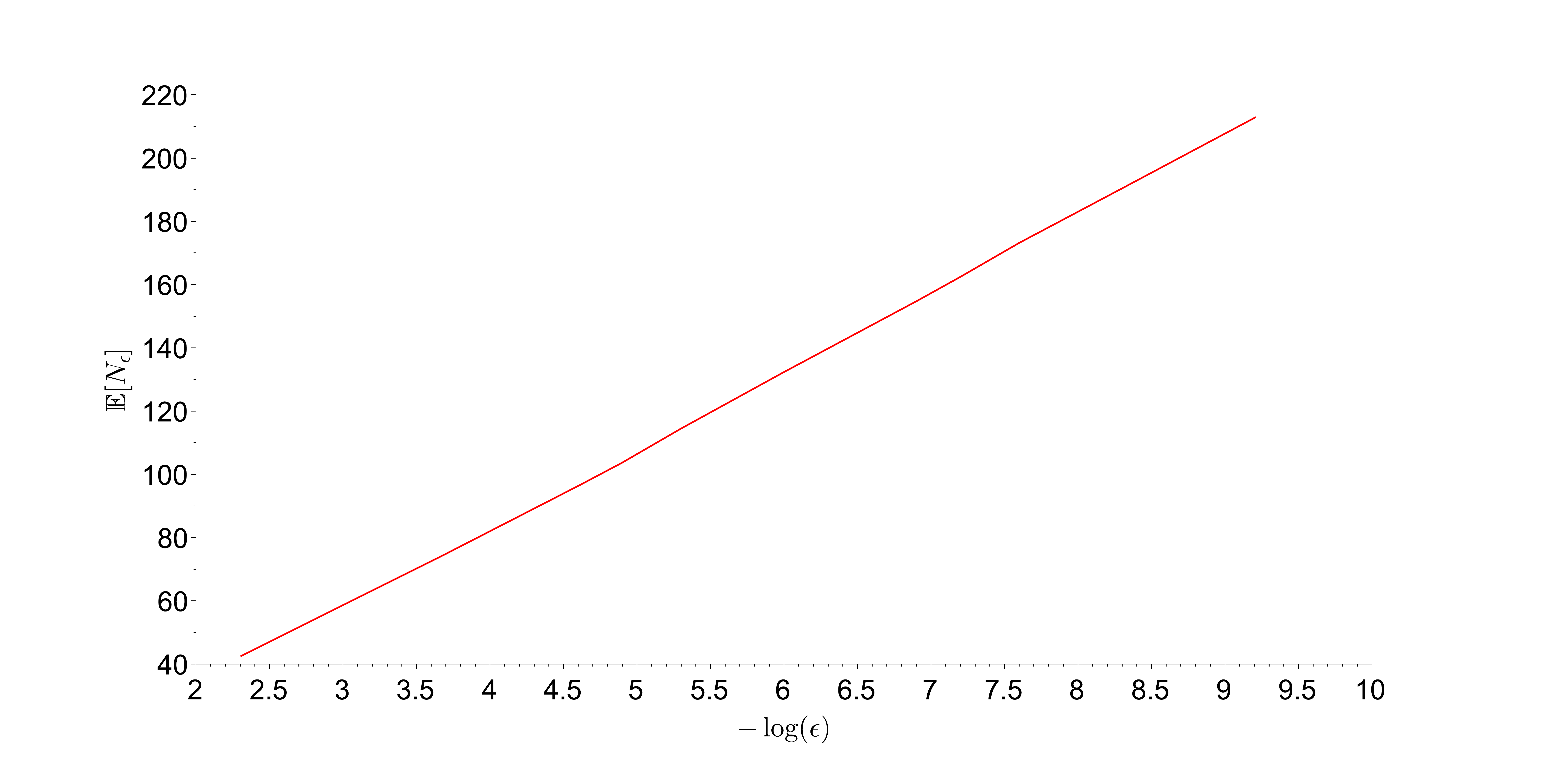}}
   \caption{\small Histogram of the outcome variable for the diffusion \eqref{sinL} with $X_0=1$, $[a,b]=[-1,2]$, $\epsilon=10^{-2}$ and $\gamma=10^{-4}$ (left). Average number of steps in {\sc Algorithm}$_m$ for the exit time of $[-1,2]$ (right, in logarithmic scale). }
   \label{sinresult}
\end{figure}
\appendix
\section{Potential theory and Markov chains}
We introduce a result coming from the potential theory and using Markov chains.\\
Let us consider a Markov chain $(X_n)_{n \in \mathbb{N}}$ defined on a state space $I$  decomposed into two distinct subsets $K$ and $\partial K$, $\partial K$ being  the  so-called frontier. Let us define $N= \inf\{n \in \mathbb{N}, X_n \in \partial K\}$ the hitting time of $\partial K$.
We assume that $N$ is a.s. finite, then the following statement holds:
\begin{proposition} 
\label{potential}
Let $G$ be a positive increasing function. 
If there exists a function $U$ such that the sequence $(H(n\wedge N, X_{n\wedge N}))_{n \in \mathbb{N}}$ is non negative and if the sequence $(H(n\wedge N, X_{n\wedge N}) + G(n\wedge N)
)_{n \in \mathbb{N}}$ represents a super-martingale adapted to the natural filtration of the considered Markov chain $(X_n)$, then  

\begin{equation*}
\mathbb{E}_x[G(N)
] \leqslant H(0,x), \quad \forall x \in K.
\end{equation*}

\end{proposition}

The proof of this classical upper-bound is left to the reader, it is essentially based on the optimal stopping theorem and on the monotone convergence theorem (see, for instance,\cite{norris}, p139).

\section{Path decomposition}
We prove in this section the two lemmas used in the proof of Proposition \ref{minderiv}. Let us just recall several notations. The process $X^x$ corresponds to the linear diffusion (3.1) with the starting value $x$; $\tau^x$ (resp. $\tau_-^x$ and $\tau_+^x$) corresponds to the first exit time of the interval $]a,b[$ (resp. $]a,b_h[$ and $]a_h,b[$), with $a_h$ and $b_h$ defined in \eqref{ahbh}.

We also recall that $(Y^\pm_t)$ stand for the solutions of the shifted SDEs \eqref{ksieds} and \eqref{majprobaTheta}. Their exit time of the interval $]a,b[$ is denoted $\mathcal{T}(Y^\pm)$ and the first passage times through the level $a$ is denoted by  $\mathcal{T}_a(Y^\pm)$.

\begin{lemma}
\label{lemmappendix}
Let $E_{ab}$ and $E_{ba}$ the two events defined by
\[
E_{ab}:=\{ \tau^x_-\le 1, X^x_{\tau^x_-}=a,\, \mathcal{T}(Y^-)\le 1, Y^-_{\mathcal{T}(Y^-)}=b \},
\] \[
E_{ba}:=\{ \tau^{x+h}_+\le 1, X^{x+h}_{\tau^{x+h}_+}=b,\, \mathcal{T}(Y^+)\le 1, Y^+_{\mathcal{T}(Y^+)}=a \}.
\]
Then $E_{ab}\cap E_{ba}=\emptyset$.

\end{lemma}

\begin{proof}
On the event $E_{ab}$ we know that $\tau_-^x \leqslant 1$ and consequently $X_s^x \in [a,b_h[ \subset [a,b[$ (for all $s<\tau_-^x$. In particular we observe that $\tau_-^x=\tau^x$. Moreover

\begin{equation*}
X^{x+h}_s = X^x_s + he^{\int^s_0 \alpha(u)du}, \quad \forall s\ge 0.
\end{equation*}
Hence
\begin{equation*}
X^{x+h}_s \in [a + he^{\int^s_0 \alpha(u)du}, b_h + he^{\int^s_0 \alpha(u)du}[, \quad \forall s\ge 0.
\end{equation*}
Since $b_h + he^{\int^s_0 \alpha(u)du} = b - h(e^{\int_0^1 \vert \alpha(u) \vert du} - e^{\int^s_0 \alpha(u)du}) < b$ for $s\le 1$, we obtain
\begin{equation*}
X^{x+h}_s \in  [a + he^{\int^s_0 \alpha(u)du}, b[\subset ]a,b[,\quad \forall s\le 1.
\end{equation*}
In conclusion $E_{ab} \subset \{\tau^x < \tau^{x+h}\}$.\\
Using similar arguments, we obtain  $E_{ba} \subset \{\tau^{x+h} < \tau^x\}$.\\
The easy observation $\{\tau^x < \tau^{x+h}\} \cap \{\tau^{x+h} < \tau^x\} = \emptyset$ implies the announced statement.
\end{proof}

\begin{lemma}
\label{lemmappendixbis}
 $E_{ab}\cup E_{ba}\subset \bigcap_{t\ge 2}\{X^{x+h}_{\tau^{x+h}\wedge t}- X^x_{\tau^x\wedge t}=b-a \}$.
 \end{lemma}
\begin{proof}
Let us prove that $E_{ab}\subset \{X^{x+h}_{\tau^{x+h}\wedge t}- X^x_{\tau^x\wedge t}=b-a \}$, the other inclusion can be obtained in a similar way.On the event $E_{ab}$ we obviously observe that $X^x_{\tau_x \wedge t} = a$.
By construction, we have $X^{x+h}_{\tau_-^x} \geqslant Y^-_0$, and using the continuity of the paths with respect to the initial condition, we obtain
$X^{x+h}_{\tau^x_- +s} \geqslant Y_s^-,\quad \forall s\ge 0.$
the property $Y^-_{\mathcal{T}(Y^-)}=b$, implies
$X^{x+h}_{\tau^x_- + \mathcal{T}(Y^-)} \geqslant Y^-_{\mathcal{T}(Y^-)}=b.$\\
Consequently $\tau^{x+h}\le \mathcal{T}(Y^-)+\tau_-^x \le 2$ and therefore, under the hypothesis $t\geqslant 2$ we have $E_{ab}\subset \{X^{x+h}_{\tau^{x+h}\wedge t}- X^x_{\tau^x\wedge t}=b-a \}$.\\
\end{proof}

\bibliographystyle{plain}

\begin{thebibliography}{}

\end{thebibliography}


\begin{thebibliography}{10}

\bibitem{bass}
R.~F. Bass.
\newblock {\em Diffusions and Elliptic Operators}.
\newblock Springer, 1997.

\bibitem{borodin-salminen}
A. Borodin and P. Salminen.
\newblock {\em Handbook of Brownian Motion - Facts and Formulae}.
\newblock Birkhauser, 1996.

\bibitem{Deaconu-Herrmann}
M. Deaconu and S. Herrmann.
\newblock Simulation of hitting times for bessel processes with noninteger
  dimension.
\newblock {\em Bernoulli, Vol. 23, no.4B pp.3744-3771}, 2017.

\bibitem{Evans}
L.C. Evans.
\newblock {\em Partial differential equations},
\newblock Graduate Studies in Mathematics, second edition, 19, American Mathematical Society, Providence, RI (2010)
  
\bibitem{Friedman}
A. Friedman.
\newblock {\em Partial differential equations of parabolic type},
\newblock Prentice-Hall, Inc., Englewood Cliffs, N.J. (1964).

\bibitem{Herrmann-Massin-1}
S.~Herrmann and N.~Massin.
\newblock {Exit problem for Ornstein-Uhlenbeck processes: a random walk
  approach};
\newblock to appear in  {\em Discrete Contin. Dyn. Syst. Ser. B}, hal-02143409, 2019.


\bibitem{lerche}
H.~R. Lerche.
\newblock {\em Boundary crossing of {B}rownian motion}, volume~40 of {\em
  Lecture Notes in Statistics}.
\newblock Springer-Verlag, Berlin, 1986.
\newblock Its relation to the law of the iterated logarithm and to sequential
  analysis.

\bibitem{mot}
M. Motoo.
\newblock Some evaluations for continuous {M}onte {C}arlo method by using
  {B}rownian hitting process.
\newblock {\em Ann. Inst. Statist. Math. Tokyo}, 11:49--54, 1959.

\bibitem{mul}
M.~E. Muller.
\newblock Some continuous {M}onte {C}arlo methods for the {D}irichlet problem.
\newblock {\em Ann. Math. Statist.}, 27:569--589, 1956.

\bibitem{norris}
J.R. Norris.
\newblock {\em Markov Chains}.
\newblock Cambridge University Press, 1996.

\bibitem{rev}
D. Revuz and M. Yor.
\newblock {\em Continuous martingales and {B}rownian motion}, volume 293 of
  {\em Grundlehren der Mathematischen Wissenschaften [Fundamental Principles of
  Mathematical Sciences]}.
\newblock Springer-Verlag, Berlin, third edition, 1999.

\bibitem{sab1}
K.~K. Sabelfeld and N.~A. Simonov.
\newblock {\em Random walks on boundary for solving {PDE}s}.
\newblock VSP, Utrecht, 1994.

\bibitem{sab2}
K.~K. Sabelfeld.
\newblock {\em Monte {C}arlo methods in boundary value problems}.
\newblock Springer Series in Computational Physics. Springer-Verlag, Berlin,
  1991.
\newblock Translated from the Russian.

\bibitem{sacerdote-telve-zucca}
L. Sacerdote, O. Telve, and C. Zucca.
\newblock Joint densities of first hitting times of a diffusion process through
  two time-dependent boundaries.
\newblock {\em Adv. in Appl. Probab., 46(1):186-202}, 2014.

\bibitem{potzelberger-wang}
L. Wang and K. Potzelberger.
\newblock Crossing probabilities for diffusion processes with piecewise
  continuous boundaries, 2008.
  
\bibitem{ikeda-wanatabe}
N. Ikeda and S. Watanabe 
  \newblock {\em Stochastic differential equations and diffusion processes.}
  \newblock volume 24, second edition
  \newblock North-Holland Publishing Co., Amsterdam; Kodansha, Ltd.,
              Tokyo, 1989.

\end{thebibliography}

\end{document}